\documentclass[reqno,12pt]{amsart}
\usepackage[utf8]{inputenc}
\usepackage[T1]{fontenc}
\usepackage{amsmath,amsfonts,amssymb,amsthm}
\usepackage{geometry,ulem}
\usepackage{tikz}
\usepackage{graphicx}
\graphicspath{{./images/}}
\usepackage[pdftex,
            colorlinks=true,
            linkcolor=blue,
            citecolor=red,
            pdfauthor={},
            pdftitle={},
            pdfcreator={pdflatex}]{hyperref}
\usepackage{verbatim}
\usepackage{color}

\newcommand{\bib}[4]{\bibitem{#1}{\sc#2: }{\it#3. }{#4.}}

\usepackage{mathrsfs}
\definecolor{verde}{RGB}{20,150,100}

\newcommand{\Z}{\mathbb Z}
\newcommand{\R}{\mathbb R}
\newcommand{\Rd}{\R^d}
\newcommand{\Rtwo}{\R^2}

\newcommand{\T}{\triangle}
\newcommand{\Tcal}{{\mathscr T}}
\newcommand{\Teq}{\T_{\text{eq}}}
\newcommand{\Topt}{\T_{\text{opt}}}
\newcommand{\Zhex}{Z^{\text{hex}}}
\renewcommand{\O}{\Omega}

\newcommand{\A}{\mathcal{A}}
\newcommand{\E}{\mathcal{E}}

\newcommand{\TT}{\mathcal T}
\renewcommand{\S}{\mathcal S}
\newcommand{\Rcal}{\mathcal{R}}

\renewcommand{\phi}{\varphi}
\newcommand{\eps}{\varepsilon}

\newcommand{\be}{\begin{equation}}
\newcommand{\ee}{\end{equation}}
\newcommand{\tq}{~ \middle | ~}
\newcommand{\btq}{~ \big | ~}
\newcommand{\Btq}{~ \Big | ~}
\newcommand{\dist}{\mathrm{dist}}
\newcommand{\per}{\mathrm{Per}}
\renewcommand{\div}{\mathrm{div}}
\renewcommand{\d}{\mathrm{d}}

\newcommand{\sm}{\setminus}
\newcommand{\vps}{\varepsilon}
\newcommand{\Om}{\Omega}
\newcommand{\sq}{\subseteq}

\newcommand{\lb}{\lambda}
\newcommand{\ov}{\overline}

\numberwithin{equation}{section}
\theoremstyle{plain}
\newtheorem{trm}{Theorem}[section]
\newtheorem{lem}[trm]{Lemma}

\newtheorem{prop}[trm]{Proposition}

\theoremstyle{remark}
\newtheorem{rem}[trm]{\bf Remark}

\title[inradius and honeycomb structures]{Shape optimization with inradius constraint and emergence of honeycomb structures}
\author{Dorin Bucur, Giuseppe Buttazzo, Alexis de Villeroch\'e}
\date{}

\begin{document}
\begin{abstract}
We maximize the average torsional rigidity among open subsets of $\mathbb{R}^2$ with fixed inradius whose complements consist of pairwise disjoint balls of fixed radius $\vps$ with mutual distances bounded below. We prove that, provided $\varepsilon$ is sufficiently small, the optimal value is asymptotically attained by sequences of domains whose complements exhibit a regular honeycomb structure. The proof is reduced, via Delaunay triangulations, to the analysis of a mixed Dirichlet--Neumann problem on triangular domains with circular cutouts at their vertices. Our approach is quite general and extends to other variational energies, such as the Cheeger constant, for which we also derive an explicit bound of the ratio $\varepsilon/$ inradius.
\end{abstract}

\maketitle

\section{Introduction}\label{Intro}

Honeycomb structures arise naturally in a wide variety of contexts. The most familiar example is the honeycomb built by bees, which inspired one of the oldest and most celebrated problems in geometry: how can the plane be partitioned into cells of equal area with the least possible total perimeter?
Known as the {\it Honeycomb Conjecture}, this problem dates back to the writings of Marcus Terentius Varro in 36 BC and was finally resolved by Hales in 2001 \cite{H01}.

The honeycomb structure emerges as a remarkably efficient configuration in a broad class of energy minimization problems arising in optimal partition and optimal distribution problems. Burdzy\footnote{Conjecture stated at an ESF research conference in Matrei, Austria, in 2005.}, Caffarelli--Lin \cite{CL07}, and Van den Berg \cite{HTOT09} independently conjectured that the regular hexagonal honeycomb should also be optimal for partitioning planar domains into cells minimizing either the sum or the maximum of the Dirichlet--Laplacian eigenvalues.
Unlike the classical {\it Honeycomb Conjecture}, this problem involves a partial differential equation on each cell of the partition. As a consequence, the optimality conditions induce a nontrivial interaction between the PDEs across the interfaces separating adjacent cells. Numerous aspects of this problem have since been extensively investigated, including the existence of optimal partitions, the regularity of the interfaces, and the qualitative properties of the cell boundaries. We refer the reader to the recent survey paper by Ognibene and Velichkov \cite{OV26} (see as well \cite{BFVV18}).

The asymptotic optimality of the honeycomb structure has so far been established only in a few specific settings, all of which are purely geometric in nature and do not involve partial differential equations. Besides Hales' theorem, examples include the Cheeger constant \cite{BuFr19}, the average distance to the boundary \cite{BB23}, and several optimal location problems in economics (see, for instance, \cite{MB01,BOS02}).
To the best of the authors' knowledge, however, no optimality result for a honeycomb structure has been established for a variational energy involving PDEs. The present paper provides such a result.

For a given open set $\O\subset\Rtwo$, we denote 
 $\rho(\O)$ the {\it inradius} of $\O$, defined as the radius of the largest disk that fits in $\O$, or equivalently as the supremum of the distance function to the boundary $\partial\O$:
\be\label{eq:inradius}
 \rho(\O)=\sup\left\{r\tq \exists x\in\O,\ B_r(x)\subset\O\right\}=\sup\left\{\dist(x,\partial\O)\tq x\in\O\right\}.
\ee
Above, $B_r(x)$ denotes the ball centered at $x$ of radius $r$. If $x=0$, then we simply denote it $B_r$.
Denoting by $\A$ a class of open subsets of $\Rtwo$, we formally study shape optimization problems writen as
\be\label{eq:SOPinrad}
\min\left\{\E(\O)\tq \O\in\A,\ \rho(\O) = 1\right\}.
\ee
Above, $\mathcal{E}$ denotes a shape functional, which in this paper will be either the torsional energy of $\Omega$ or the Cheeger constant. The optimization problem is clearly ill-posed if $\mathcal{A}$ is taken to be the class of all open sets. Indeed, the prescribed inradius plays no effective role for either the torsional rigidity or the Cheeger constant: removing or adding any countable set of points does not affect these functionals, since such sets have both zero Lebesgue measure and zero capacity, whereas it may drastically alter the inradius.

In the literature, Problem \eqref{eq:SOPinrad} has been extensively studied under the additional assumption that the admissible domains $\O$ are simply connected, primarily when $\E$ is the first Dirichlet eigenvalue of the Laplace operator (see \cite{BBO24} and the references therein). A distinctive feature of this optimization problem is that even the existence of an optimal domain remains a highly nontrivial issue (see \cite{BaCa11} for the analogous problem of maximizing the $L^\infty$-norm of the torsion function). Moreover, it has long been conjectured that, for a broad class of such shape functionals, the optimal value is asymptotically attained by the Goodman domain $\mathcal G$ (see \cite{BC94,G45}), which is obtained by removing infinitely many radial slits from the plane (see Figure~\ref{fig:Goodman}). By contrast, the Hersch pipe, which was also conjectured to be optimal for minimizing the first Dirichlet eigenvalue, was recently shown not to be optimal \cite{BBO24}.
\begin{figure}[h]
    \centering
    \includegraphics[scale = 0.4]{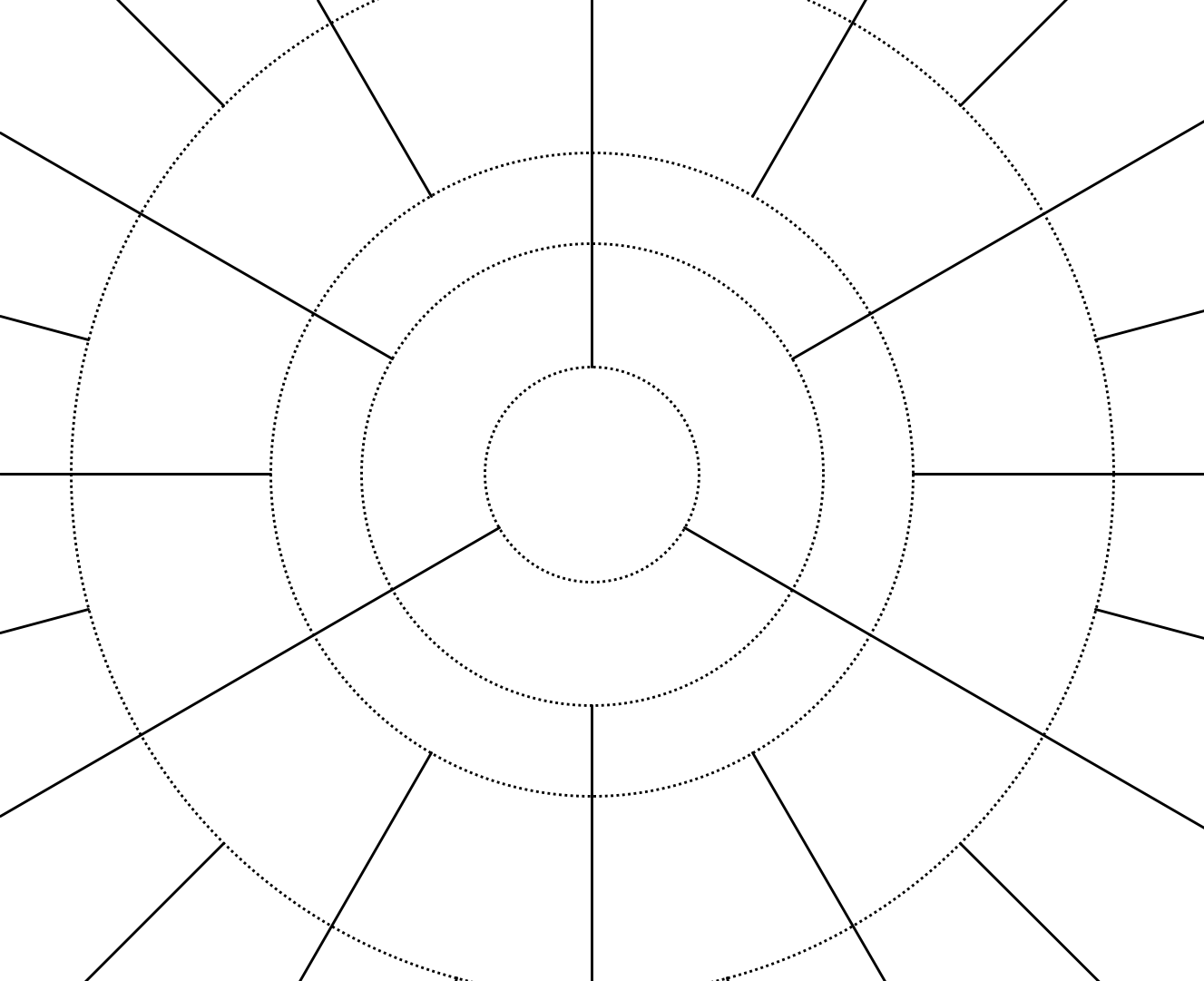}
    \hskip 1.8cm
    \includegraphics[scale = 0.25]{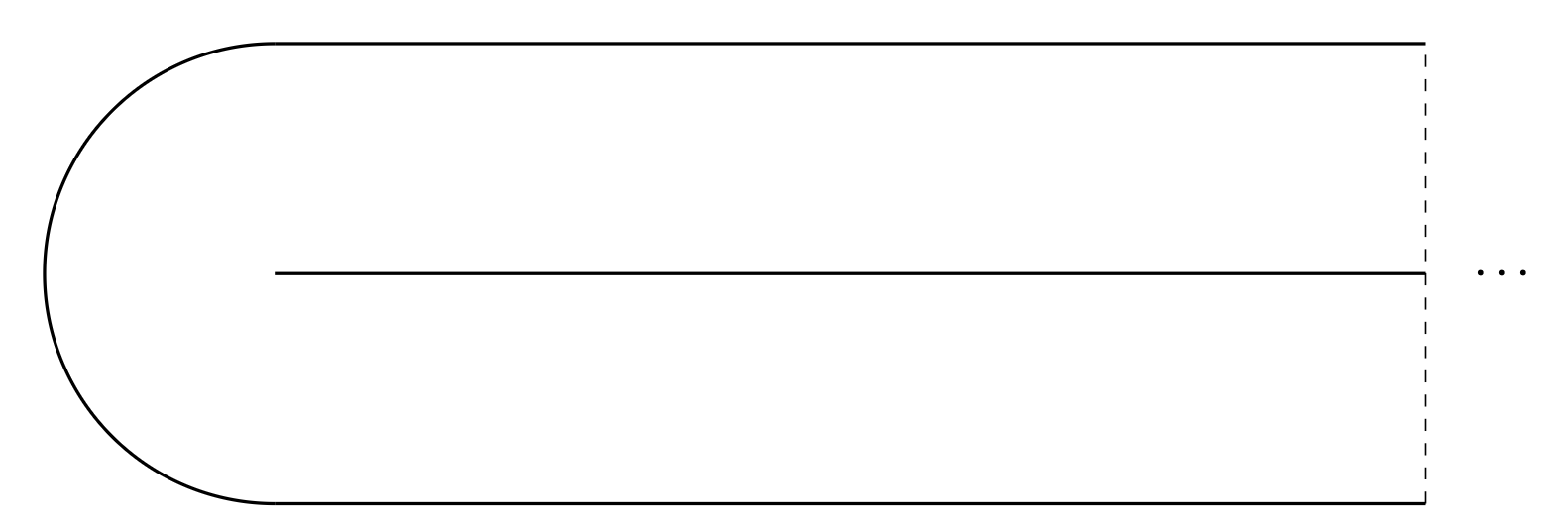}
    \caption{Goodman domain \label{fig:Goodman} and Hersch pipe}
\end{figure}

If simple connectivity is not required and $\E$ is a variational energy associated with the Dirichlet $p$-Laplacian (with $p>2$, so that points have positive $p$-capacity), one may expect the complements of optimal domains to be very small, possibly consisting of isolated points. For instance, it was shown in \cite{BBdV} the existence of a minimizer of the first Dirichlet eigenvalue of the $p$-Laplacian (for $p>2$) which is the complement of a discrete subset of the plane. Although the precise arrangement of these points remains unknown, it is conjectured that they form a regular hexagonal lattice. Such a characterization has so far been established only for the average $\infty$-torsional rigidity in \cite{BB23}, a problem that is again purely geometric in nature. In that setting, when $\rho=1$, the optimal domain is, in a suitable asymptotic geometric sense (namely, as the limit of a minimizing sequence), formally given by $\O^*=\Rtwo\setminus\Zhex$ where $\Zhex$ denotes the set of centers of the hexagonal tiling
\be\label{eq:Zhex}
\Zhex = \left\{\sqrt3\left(n+\frac m2,\frac{\sqrt3 m}{2}\right):
 n,m\in\mathbb Z\right\}.
\ee

In this paper we focus on the case $p=2$, which is critical for the capacity: this is the largest value of $p$ for which points have $0$ $p$-capacity. For this reason, in order to make the problem relevant, we ask the competing domains to satisfy an exterior $\eps$-ball condition, for some prescribed $\eps >0$. We work in the class 
$$\A_\vps=\{\Om \sq \R^2 ~|~ \Om \mbox { open and }\R^2\sm \Om = \ov {\cup_\alpha B_\eps(x_\alpha)}\}$$
$$\hskip 1cm=
\left\{\O = \Rtwo\setminus \big( Z \oplus \overline{B_\eps}\big) \tq Z \subset \Rtwo, Z \mbox { closed} \right\}.$$
For a given open set $\O$, we define the associated torsional rigidity $T(\O)$ as 
$$
T(\O) = \sup \left\{ \frac{\left(\int_\O u\right)^2}{\int_\O |\nabla u|^2} \tq u\in H^1_0(\O)\right\},
$$
which may be equivalently written 
\be \label{eq:deftorsion}
T(\O)= \sup\left\{2\int_\O u-\int_\O|\nabla u|^2\tq u\in H^1_0(\O)\right\}.
\ee
If $\O$ has finite measure, we denote its average by 
$$\E(\O) = \frac{T(\O)}{|\O|}, $$
and if its measure is infinite, we set
$$\E(\O):=\limsup _{R \to \infty} \sup_{x\in \R^2} \frac{T(\O\cap B_R(x))}{|\O\cap B_R(x)|}.$$
The main purpose of the paper is to solve the following problem,
\be \label{pb:torsion}
\sup \left\{\E(\O) \tq \O \in \A_\eps,\ \rho(\O) \le 1\right\},
\ee
and to prove that the supremum is attained for a set $\O$ of inradius equal to $1$, corresponding to a regular lattice $Z$ of hexagonal type. Clearly, in view of the definition of the average torsional rigidity of an unbounded open set, it is enough to consider 
the supremum over bounded open sets which satisfy the inradius constraint and the outer $\eps$-ball condition.

Note that the supremum in \eqref{pb:torsion} is finite since the torsion function of any bounded set $\O \in \A_\eps$ is bounded from above by a constant depending only on $\eps$ (and the upper bound of the inradius, equal to $1$ in our setting). This is a consequence of the fact that the first Dirichlet eigenvalue of such a domain is uniformly bounded from below by a constant depending on $\eps$ only and of the Van den Berg inequality: there exists a constant $C$ such that for every (bounded or unbounded) open set $\Om \sq \R^2$
$$\|w_\O\|_\infty \le \frac{C}{\lb_1(\Om)}.$$

Yet, since the functional $\E$ is not increasing for the inclusion of sets we are not able to give a full answer to Problem~\ref{pb:torsion} and we must restrict the class of admissible shapes to 
$$ \ov \A_\eps =\{\O\subset \Rtwo~ |~ \Rtwo \sm \O = Z\oplus B(0,\eps),\ Z\text{ discrete }, \forall p,q \in Z,\ |p-q| \ge \sqrt {3\eps}\}. $$
In Figure~\ref{fig:admopt}, we give an example of a set in $\ov \A_\eps$.

Our main result is the following.

\begin{trm} \label{TH1}
There exists $0 < \eps^* < \frac12$ such that for all $0<\eps < \eps^*$, the set 
$$\R^2 \sm [(1+ \eps)\Zhex\oplus\overline{B_\eps}]$$
 is a solution for the problem 
$$\sup \left\{\E(\Om) \tq \O \in \ov\A_\eps,\ \rho(\O) \le 1\right\}. $$
\end{trm}
 
\begin{figure}[ht]
\centering
\includegraphics{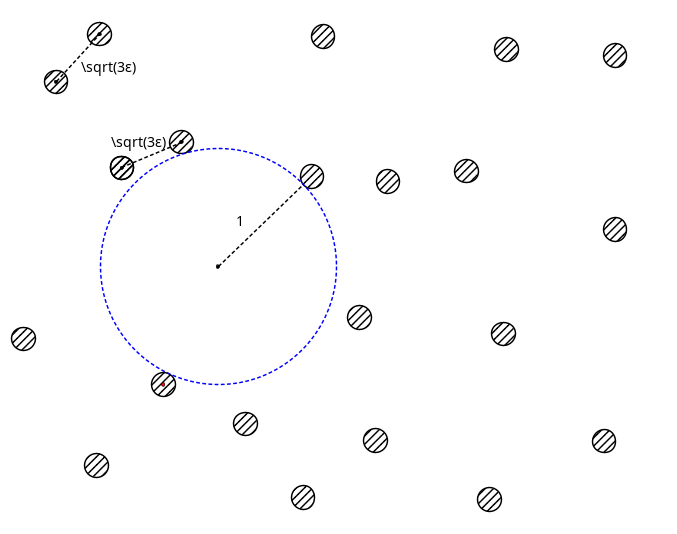}
\hspace{5em}
\includegraphics{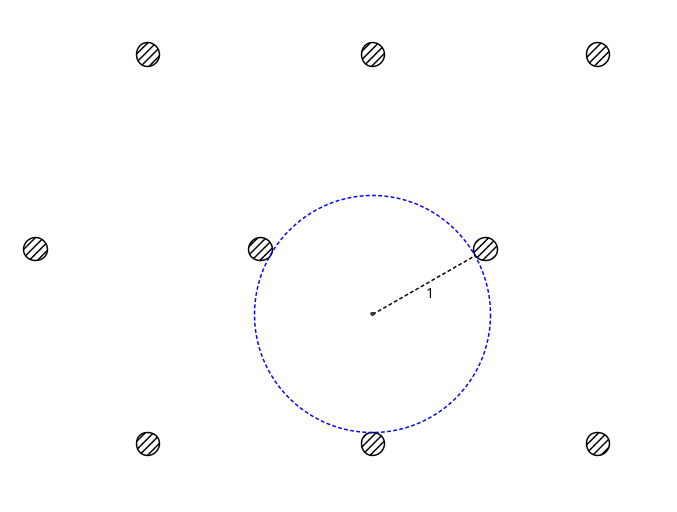}
\caption{admissible set in $\ov\A_\eps$ (left) and optimal set (right)}
\label{fig:admopt}
\end{figure}

Above, $\Zhex$ is the set of the centers of the hexagonal tessellation of circumradius $1$, defined in \eqref{eq:Zhex}. To prove this result, we follow, for the geometric part of the analysis, the strategy developed in \cite[Theorem 1.2]{BB23}. In contrast with \cite[Theorem 1.2]{BB23}, however, the main new difficulty lies in handling the variational energy itself, which is defined through the solution of an elliptic PDE, namely the torsion function. More precisely, the first part of the proof closely parallels that of \cite[Theorem 1.2]{BB23} and consists in reducing the original optimization problem over general domains to the analysis of a mixed boundary value problem on triangles $\T$ with small discs removed around their vertices. The second part is fundamentally different: instead of dealing with a purely geometric functional on the triangles, we must analyze an elliptic PDE with mixed boundary conditions.

Precisely, we generically denote by $\T$ a triangle of vertices $V_\T=\{A_1, A_2, A_3\}$. Given $r >0$, we introduce the mixed torsional rigidity of $\T \sm [V_\T\oplus \ov B_r]$, by
\be\label{eq:defT*}
\Tcal (\T,r)=\sup\left\{2\int_\T u\,dx -\int_\T |\nabla u|^2\,dx \tq u \in H^1(\T), \ \forall i,\ u = 0 \text{ on } \T\cap [V_\T\oplus \ov B_r]\right\},
\ee
which involves Dirichlet boundary conditions on the boundary of the small disks surrounding each vertex, and Neumann boundary conditions of the remainder of the boundary of the triangle. 
When $|\T\sm [V_\T\oplus \ov B_r]| = 0$ we have $\Tcal (\T,r)=0$ and when $|\T\sm [V_\T\oplus \ov B_r]| > 0$, the associated torsion function $w_r$ solves 
$$\begin{cases}
-\Delta w_r = 1 & \text{in } \T\setminus [V_\T\oplus \ov B_r]\\
w_r = 0 & \text{on } \T\cap[V_\T\oplus \ov B_r]\\
\partial_n w_r = 0 & \text{on } \partial\T\setminus[V_\T\oplus \ov B_r].
\end{cases}$$
We naturally have
$$\Tcal(\T,r) = \int_\T w_r\,dx>0.$$

In the following theorem, we formalize the relationship between \eqref{pb:torsion} and the mixed torsion problem on punctured triangles.   It is for this step that we need the restricted class and consider $\ov\A_\eps$ instead of $\A_\eps$, specifically to prove the the second inequality.  
\begin{trm}\label{Th:reduction}
Let $\Teq$ denote the equilateral triangle of circumradius $1$ and
assume $0 < \eps <\frac 12$. Then, we have the following
$$
\frac{\Tcal((1+\eps)\Teq, \eps)}{(1+\eps)^2|\Teq| - \frac{\pi}{2}\eps^2}\le \sup \left\{\E(\Om) \tq \O \in  \ov\A_\eps  ,\ \rho(\O) = 1\right\} $$
$$
\le\sup\left\{\frac{\Tcal (\T, \eps)}{| \T \sm [V_{\T} \oplus B_\eps]|} \tq\T \text{ triangle},\Rcal(\T) \le 1+\eps, |\T \sm [V_{\T} \oplus B_\eps]| >0 \right\}.
$$
\end{trm}

As a consequence, the proof of Theorem~\ref{TH1} comes (up to rescaling) to solving the maximization problem on triangles
\be\label{pb:triangle}
\sup\left\{\frac{\Tcal(\T,r)}{|\T\sm [V_\T \oplus B_r]|} \tq \T \mbox{ triangle},\Rcal(\T) \le 1, |\T \sm [V_{ \T} \oplus B_r]| >0\right\}
\ee
and to prove that the solution is the equilateral triangle $\Teq$ of circumradius $1$.

Even if we do not use it, it is convenient to observe that if $|\T \sm [V_{\T} \oplus B_r]|=0$, it is natural to formally set
$$ \frac{\Tcal(\T,r)}{|\T\sm [V_\T \oplus B_r]|} =0, $$
as a consequence of the behavior of the mapping $r\to \frac{\Tcal(\T,r)}{|\T\sm [V_\T \oplus B_r]|}$ at $0$ (this fact is implicitly proved in Lemma \ref{d71}). 

As well, notice that in general
$$ |\T\sm [V_\T \oplus B_r]| \ge |\T| - \frac{\pi r^2}{2}.$$
For right or acute triangles $\T$, equality occurs when each edge is not smaller than $2r$. If $\T$ is obtuse, for equality to occur we also need that every height is not smaller than $r$. We call these triangles {\it non degenerate}.
 
The key technical result of the paper is the following.
\begin{trm}\label{TH2}
There exists $r^*>0$ such that for all $0<r<r^*$, the equilateral triangle $\Teq$ of circumradius $1$ is the unique solution to Problem~\eqref{pb:triangle}.
\end{trm}

Similar questions naturally arise for the Dirichlet eigenvalues of the $p$-Laplacian with $p\in[1,2]$. In fact the reduction argument stated in Theorem~\ref{Th:reduction} is fairly general and extends to a wide range of variational shape functionals, such as the first Dirichlet eigenvalue of the $p$-Laplacian or the averaged $p$-torsional rigidity, for every $p \in (1, +\infty)$. Therefore, proving the optimality of the honeycomb distribution for these functionals ultimately amounts to establishing a suitable counterpart of the triangular inequality, in the spirit of Theorem~\ref{TH2}.

To underline the main lines of this problem and compare it with the torsion energy, we conclude the paper by briefly discussing the case of the Cheeger constant, which formally corresponds to the limiting case $p=1$. For an arbitrary open set $\Om$, the Cheeger constant of $\O$ is defined by
\be\label{eq:defcheeger}
h(\O)=\inf\left\{\frac{\per(E)}{|E|} \middle|\ E\mbox{ measurable, }E\Subset\O\right\},
\ee
where $\per(E)$ denotes the perimeter of $E$ in the sense of De Giorgi. While we obtain the same type of result as in Theorem \ref{TH1}, our setting also allows us to derive explicit constants, thanks to the purely geometric nature of the problem.

\begin{trm}\label{Th:cheeger}
For any $0<\eps<0.23$, the open set 
$$\O^* = \Rtwo\setminus \big((1+\eps)\Zhex\oplus \overline{B_\eps}\big),$$
is a solution for the problem 
$$\min\big\{h(\O)~|~\rho(\O) = 1,\ \O\in \A_\eps\big\}.$$
\end{trm}

  Note that in this case it is not needed to reduce the class of admissible shapes as the Cheeger constant is decreasing for the inclusion on sets.  

\begin{rem}
The value of $\eps^*=0.23$ above is explicit, but not sharp. Moreover, for $\eps\le \eps^*$ the value $h(\O^*)$ is explicit,
$$h(\O^*) = \frac{4\pi\eps}{3\sqrt3(1+\eps)^2-2\pi\eps^2}. $$
We expect that Theorem \ref{Th:cheeger} should hold for any $\eps>0$. 
\end{rem}

\noindent {\bf Structure of the paper.} The proof of Theorem \ref{TH1} is a direct consequence of Theorems \ref{Th:reduction} and \ref{TH2}. Section~\ref{sec:proofTHreduc} is devoted to the proof of Theorem~\ref{Th:reduction}. Section~\ref{sec:proofTH2} gives the lines of the proof of Theorem~\ref{TH2} and splits it in a series of three propositions proved, respectively, in Sections \ref{sec:proofPropConv}, \ref{sec:proofPropMaxF0} and \ref{sec:proofPropMaxFr}. Section~\ref{sec:Cheeger} contains the proof of Theorem~\ref{Th:cheeger}. 

\section{Proof of Theorem~\ref{Th:reduction}}\label{sec:proofTHreduc}

\noindent{\bf Delaunay triangulations and Voronoi diagrams.} In order to split the torsion energy in cells that we can handle individually, we rely on Delaunay triangulations \cite{D34} of clouds of points. Let us first recall their definition and some of their properties; we refer the interested reader to the monograph \cite{AKL13} for more details.

 Consider a discrete set of points $S$ in $\Rtwo$ (possibly infinite, with no accumulation). For every $p\in S$, the associated Voronoi cell $V(p)$ is 
$$ V(p) = \big\{x\in\Rtwo\btq\forall q\in S\setminus\{p\},\;|x-p| < |x-q|\big\}.$$
For a given $p\in S$, the cell $V(p)$ is built as follows: for every $q\in S\setminus \{p\}$, consider $d_{p,q}$ the mediator line of the segment $[p,q]$ which divides $\Rtwo$ into to half planes, denoted by $H_{pq}^+$ for the one containing the point $p$. With this notation, $V(p)$ then becomes
$$V(p) = \bigcap_{q\in S\setminus \{p\}} H_{pq}^+,$$
and consequently, it is a convex set with polygonal boundary (possibly unbounded). The decomposition of the plane following this rule is the Voronoi diagram of the set $S$, further denoted by $\mathcal V(S)$ (Figure~\ref{fig:Voronoi}, left).
\begin{figure}[h]
\centering
\includegraphics[scale = 0.38]{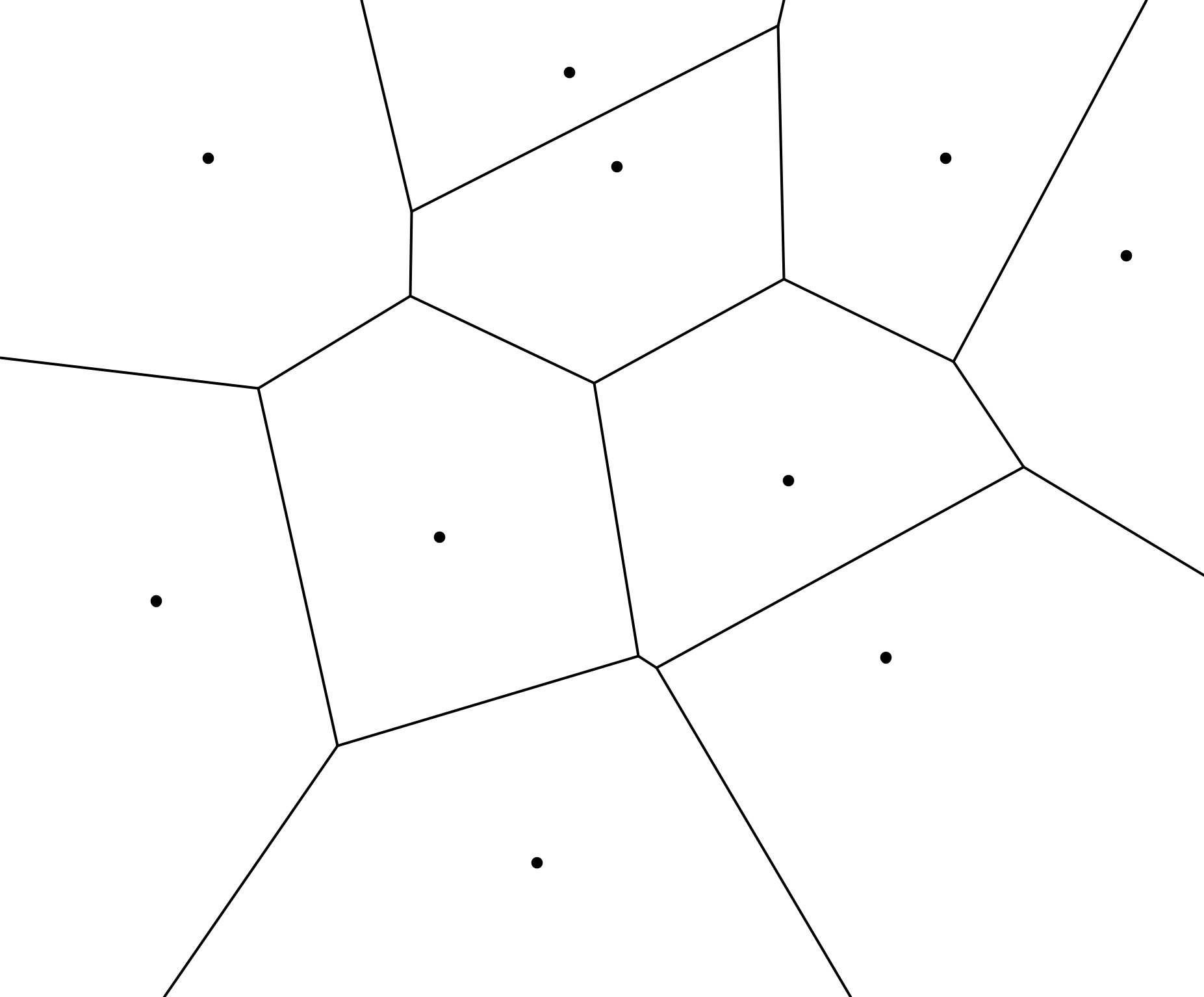}
\includegraphics[scale = 0.4]{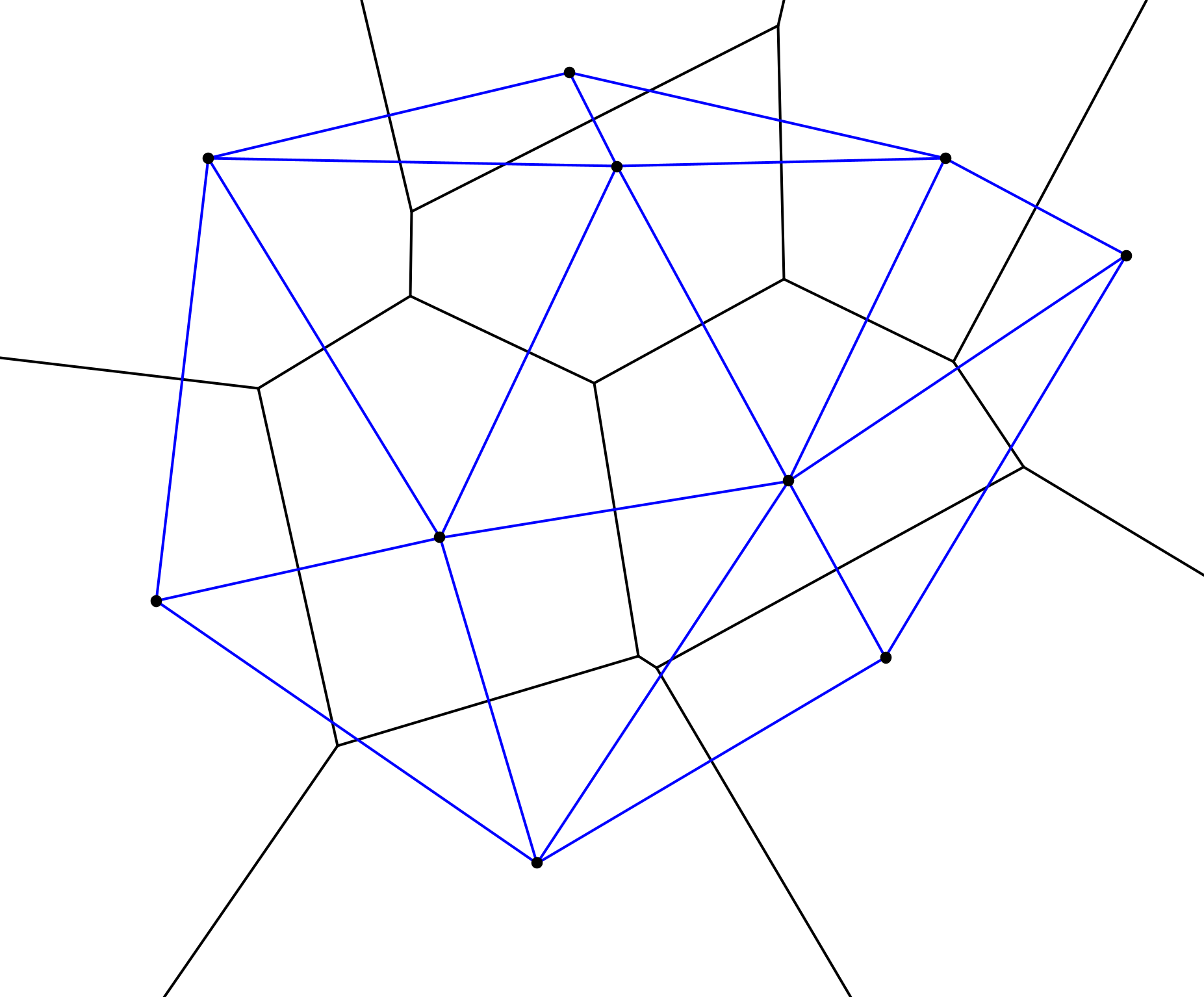}
\caption{Voronoi Diagram and its Delaunay triangulation \label{fig:Voronoi}}
\end{figure}

Associated to this Voronoi diagram one defines the Delaunay tessellation $\mathcal D(S)$ of the set $S$ as a graph: for every two points $p,q\in S$, the segment $[p,q]$ is an edge in the graph $\mathcal D(S)$ if and only if the two associated Voronoi cells $V(p)$ and $V(q)$ are adjacent. In general, the cells defined by the Delaunay tessellation of $S$ are not triangular, but up to suitably adding some edges to the graph $\mathcal D(S)$ that do not cross the existing ones, we obtain what is known as a Delaunay triangulation of $S$ that we denote in the following as $\TT(S)$ (Figure~\ref{fig:Voronoi}, right). 
 
The main property of $\TT(S)$ on which our proof heavily relies is the so-called {\it empty circle property} (see \cite{D34}), stating that 
 if $\T$ is a triangle in $\TT(S)$ and $\mathcal C_\T$ is its circumscribed circle, then inside the circle there is no other point of $S$.

We are now ready to prove Theorem~\ref{Th:reduction}.

\begin{proof}[Proof of Theorem \ref{Th:reduction}]

\smallskip
Let us prove now the first inequality of Theorem \ref{Th:reduction}.
Let $\eps \in (0, \frac 12)$ and consider $(1+\eps) \Zhex$ the set of centers of the hexagonal tessellation build on equilateral triangles of circumradius $1+\eps$. We denote by $\TT^\text{hex}$ its Delaunay triangulation which consists precisely on equilateral triangles of edge of length $1+\eps$, and define the set $\O^* = \Rtwo \setminus \big((1+\eps)\Zhex\oplus \overline{B_\eps}\big)$, we have $\O\in\ov\A_\eps$. We construct the sequence of sets $(\O_n)_{n\ge 1}$ defined by $\O_n = \O^* \cap B_n$.
Then, $|\O_n| < +\infty$ and if $n \ge 2$ we have $\rho(\O_n) = 1$   and by definition, it satisfies 
$$ \limsup \frac{T(\O_n)}{|\O_n|} \le \E(\O^*) \le \sup\{\E(\O)~|~ \O\in \ov \A_\eps, \rho(\O) = 1\}.$$
  Let $M_n$ be the number of triangles of $\T\in \TT^\text{hex}$ such that $\T \subset B_n$ and $d(\T, \partial B_n) \ge 1$. Denote by $\TT_n^\text{hex}$ the family of such triangles. When $n\to +\infty$, $M_n$ behaves like $O(n^2)$. Consequently, $|\O_n| = M_n ((1+\eps)^2|\Teq| - \frac{\pi \eps^2}{2}) + O(n)$. 

To estimate from below the torsion of $\O_n$, we build a test function $u_n\in H^1_0(\O^*\cap B_n)$ such that for all triangles $\T\in\TT_n^\text{hex}$ we have
$$\int_\T 2\,u_n\,dx-\int_\T|\nabla u_n|^2\,dx = \Tcal((1+\eps)\Teq, \eps)$$
and $u_n$ has a controlled gradient, uniformly bounded, at distance one from the boundary of the ball $ B_n $.

This construction is possible because the mixed torsion function associated to $(1+\eps)\Teq \sm [V_{(1+\eps)\Teq}\oplus \ov {B_ \eps}]$ enjoys the symmetry of the equilateral triangle. In particular, on two adjacent triangles the associated torsion functions have the same traces on the joint boundary so that they can be glued together in order to make a global $H^1_{\text{loc}}$ function. We consider $w_{\Om^*}$ the function defined in this way on $\O^*$ and define
$$u_n:= \min\{w_{\Om^*}, \|w_{\Om^*}\|_\infty d(\cdot, \partial B_n)\}.$$
This function differs from $w_{\Om^*}$ only in a tubular neighborhood of distance $1$ to $\partial B_n$ being unchanged on triangles from $\TT_n^\text{hex}$. In the $1$-tubular neighborhood of the boundary $\partial B_n$, $u_n$ has controlled $L^\infty$-norms of both the function itself and its gradient, by constants independent on $n$. 

We then obtain
$$T(\O_n) \ge M_n \,\Tcal((1+\eps)\Teq, \eps) + O(n)$$ 
so that
$$\frac{T(\O_n)}{|\O_n|}\ge\frac{\Tcal(( {1+\eps})\Teq, \eps)}{( {1+\eps})^2|\Teq| - \frac \pi 2\eps^2} + O(1/n),$$
and
$$\lim_{n\to\infty} \frac{T(\O_n)}{|\O_n|} \ge \frac{\Tcal(( {1+\eps})\Teq, \eps)}{( {1+\eps})^2|\Teq| - \frac \pi 2\eps^2},$$
which concludes the proof of the first inequality.

\bigskip
We prove now the second inequality.

Fix $\eps \in(0,\frac 12)$ and let us consider $\O \in \ov \A_\eps$, $\O = \Rtwo \setminus \big(Z \oplus \overline{B_\eps}\big)$, satisfying $\rho(\O) = 1$.   and define $\O_R = \O \cap B(x,R)$ for some $R>0$ and $x\in \Rd$.
Denote by $\overline\TT(Z) \subset\TT(Z)$ the subset of triangles $\T$ of $\TT(Z)$ satisfying $|\T\cap \O_R| > 0 $, and we introduce the set $\O_R^{\text{ext}} = \O_R \cup \bigcup_{\T\in\ov\TT(Z)}\T\cap \O$. Note that since the points of $Z$ must stay at distance at least $\sqrt{3r}$, the volumes $|\O_R|$ and $|\O_R^{\text{ext}}|$ satisfy uniformly for $x\in \Rd$,
$$ \lim_{R\to +\infty}\frac{|\O_R^{\text{ext}}|}{|\O_R|} = 1. $$ 

Since $Z$ has no accumulation points and $\O$ is bounded, the set $\overline \TT (Z)$ is finite. Denoting by $w^\text{ext}$ the torsion function of $\O_R^\text{ext}$, we have
$$
\begin{aligned}
\frac{T(\O_R)}{|\O_R|} & \le \frac{|\O_R^{\text{ext}}|}{|\O_R|} \frac{T(\O^\text{ext}_R)}{|\O^\text{ext}_R|}\\
&= \frac{|\O_R^{\text{ext}}|}{|\O_R|} \frac{\sum_{\T\in \overline \TT(Z)} \left[\int_\T 2w^\text{ext} - \int_\T |\nabla w^\text{ext}|^2\right]}{\sum_{\T\in\overline \TT(Z)} |\O\cap\T|}\\
&\le \frac{|\O_R^{\text{ext}}|}{|\O_R|} \sup_{\T\in\overline \TT(Z)} \frac{\int_\T 2 w^\text{ext} - \int_\T |\nabla w^\text{ext}|^2}{|\O\cap\T|}\\
&\le \frac{|\O_R^{\text{ext}}|}{|\O_R|} \sup_{\T\in\overline \TT(Z)} \frac{ \Tcal(\T, \eps)}{|\O\cap\T|}.
\end{aligned}
$$
Let us denote $\T^*$ the triangle reaching the supremum above. Since $\TT(Z)$ is a Delaunay triangulation and $\rho(\O) =1$, by the empty circle property, then $\T^* \in \TT(Z)$ has circumradius $\Rcal(\T^*)\le 1+\eps$. Additionally, since $\O\in \ov\A_\eps$ and $\eps < \frac12$, it holds that $|\O\cap\T^*| = |\T^*\sm [V_{\T^*}\oplus \ov{B_\eps}]|$, so that

$$ \frac{ \Tcal(\T^*, \eps)}{|\O\cap\T^*|} = \frac{ \Tcal(\T^*, \eps)}{|\T^*\sm [V_{\T^*}\oplus \ov{B_\eps}]|}. $$

Note that this would not be true in general if we were to work in the class $\A_\eps$. Indeed, if no condition is imposed on the distance between points, the disk positioned on the third vertex of a triangle adjacent to $\T^*$ could overlap with $\T^*$ resulting in a loss of volume. 

Finally, taking the the sup over $x$ and the limsup in $R$ concludes the proof of the second inequality.
 
\end{proof}

\section{General strategy and main steps of the proof of Theorem~\ref{TH2}}\label{sec:proofTH2}

In this section we detail the proof strategy of Theorem~\ref{TH2}. It relies on three distinct steps, stated in Propositions \ref{prop:Convergence}, \ref{prop:maximalityF0} and \ref{prop:MaxFr}. These propositions are proved in Sections \ref{sec:proofPropConv}, \ref{sec:proofPropMaxF0} and \ref{sec:proofPropMaxFr}, respectively.

\begin{proof}[Proof of Theorem \ref{TH2}]
We are looking to solve, for $r$ small, problem \eqref{pb:triangle}, which we recall for the convenience of the reader:
\be \label{d70}
\sup\left\{\frac{\Tcal(\T,r)}{|\T\sm [V_\T \oplus B_r]|} \tq \T \mbox{ triangle},\Rcal(\T) \le 1, |\T \sm [V_{ \T} \oplus B_r]| >0\right\}
\ee
 
\begin{lem}\label{d71}
There exists $r^*>0$ such that, for every $r\in(0,r^*)$, Problem~\eqref{d70} admits a solution $\T^*$. Moreover, every optimal triangle $\T^*$ is either acute or isosceles right, is non degenerate, and has circumradius $\Rcal(\T^*) = 1$. 
\end{lem}
\begin{proof}
Let us first argue that if there is a solution to Problem \eqref{d70}, it must be an acute or a right isosceles triangle.
\begin{figure}[h]
\centering
\includegraphics[scale = 0.2]{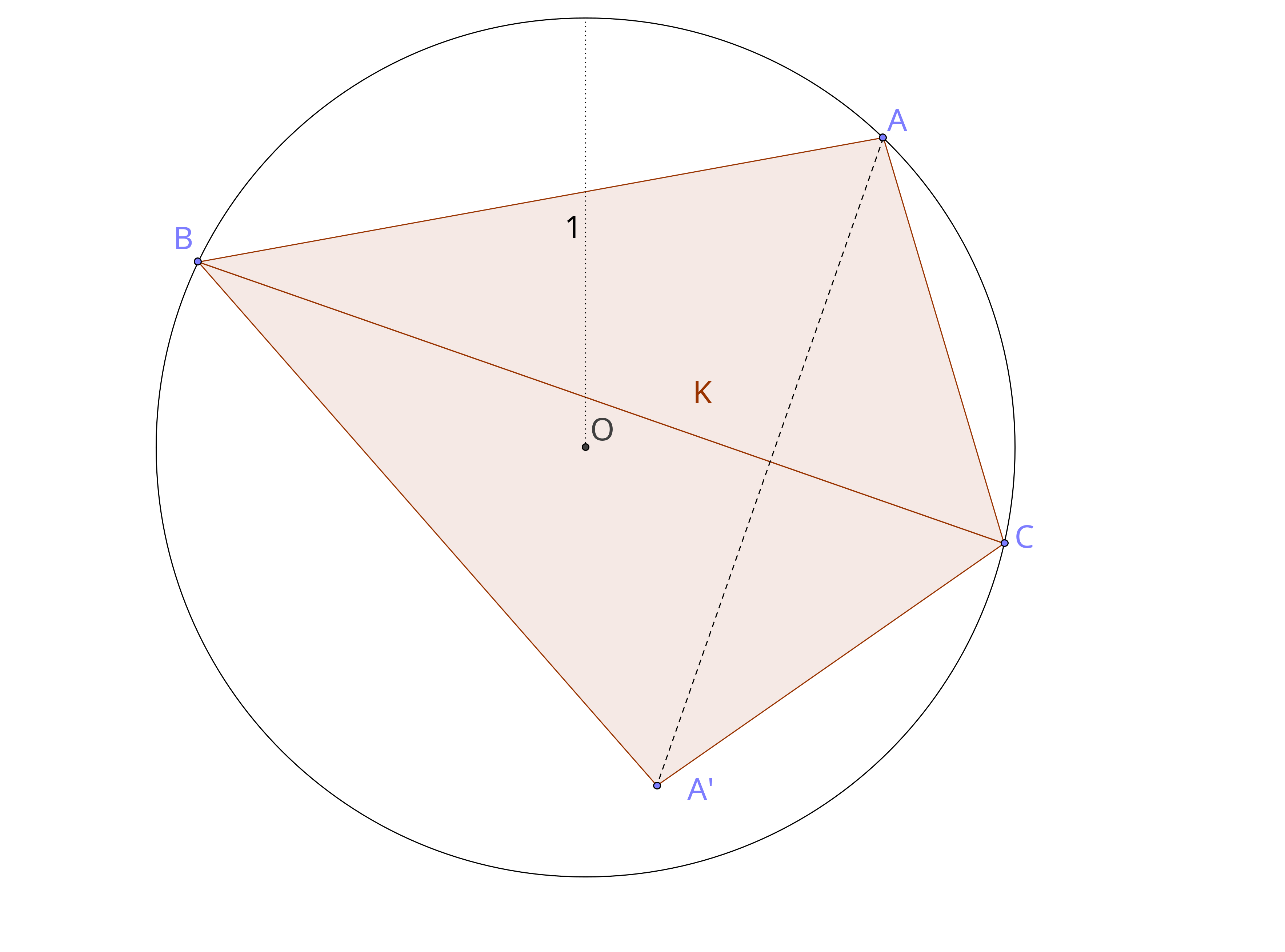}
\includegraphics[scale = 0.2]{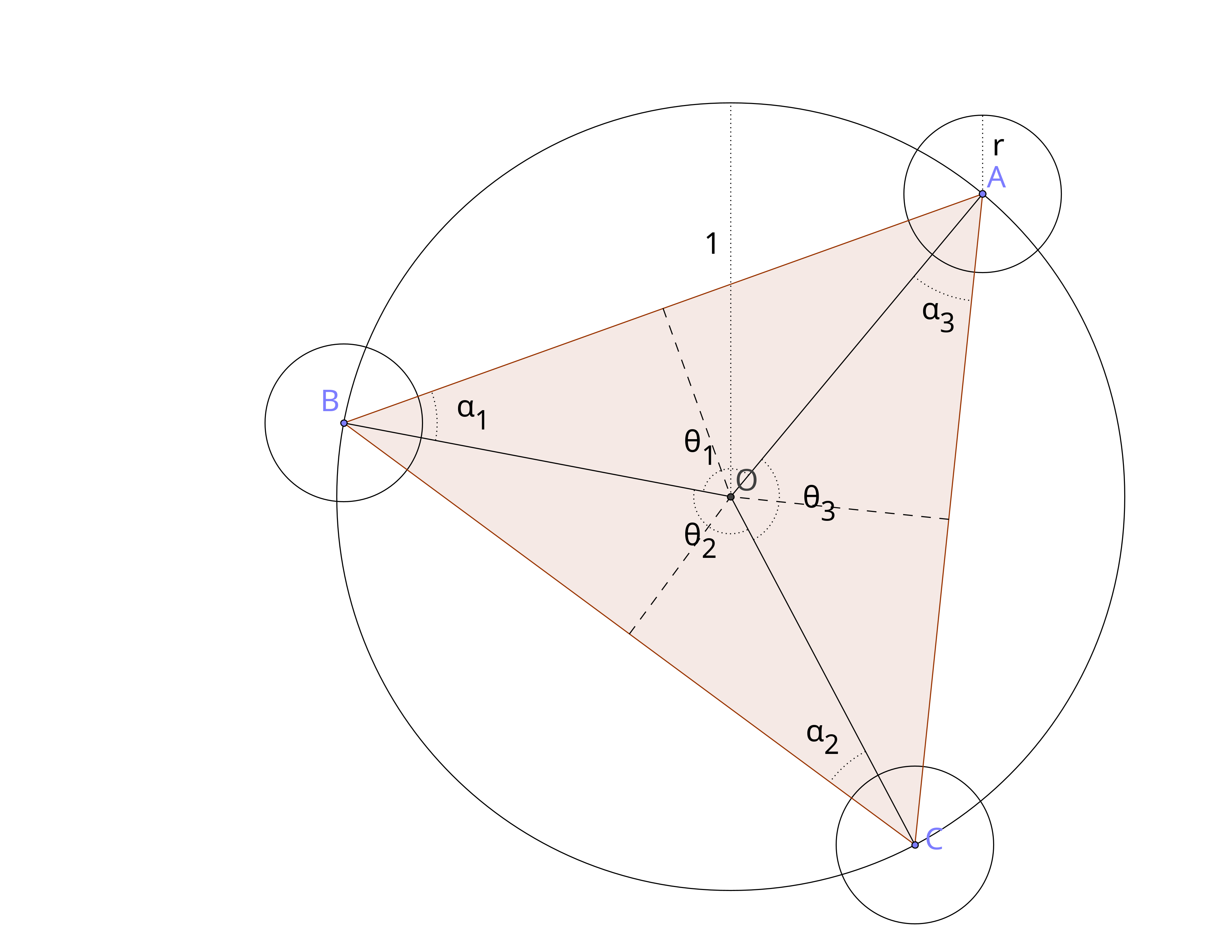}
\caption{Symmetrization procedure and division into subtriangles \label{fig:symmetrization}}
\end{figure}
Let $\T(A,B,C)$ be a triangle of vertices $ABC$ which admits an obtuse angle in $A$ of circumradius $\Rcal(\T(A,B,C))\le 1$ and such that $|\T(A,B,C)\sm [V_{\T(A,B,C)}\oplus \ov{B_r}]|>0$. We perform the following construction (Figure~\ref{fig:symmetrization}, left). Consider the kite $K$ of vertices $A,B,A',C$ obtained by reflection of the triangle with respect to the side $BC$. Since the angle in $A$ is obtuse, the point $A'$ is included in the circumscribed circle of $\T(A,B,C)$ which implies that the triangles $\T(A,B,A')$ and $\T(A',C,A)$ both have circumradius not larger than $1$. Relying on symmetry, we obtain
$$
\frac{\Tcal(\T(A,B,C), r)}{|\T(A,B,C)\sm [V_{\T(A,B,C)}\oplus B_r]|}$$
$$=\frac{\Tcal(\T(A,B,C), r)+\Tcal(\T(A',B,C), r)}{|\T(A,B,C)\sm [V_{\T(A,B,C)}\oplus \ov{B_r}]| + |\T(A',B,C)\sm [V_{\T(A',B,C)}\oplus \ov{B_r}]|}$$
$$\le \max\left\{\frac{\Tcal(\T (A,B,A'),r)}{|\T(A,B,A')| \sm [V_{\T(A,B,A')}\oplus \ov{B_r}]|},\ \frac{\Tcal(\T (A',C,A),r)}{|\T(A',C,A)| \sm [V_{\T(A',C,A)}\oplus \ov{B_r}]|}\right\}.
$$
Therefore the best of the two triangles, $\T(A,B,A')$ and $\T(A',C,A)$ is not a worst competitor than $\T(A,B,C)$. Moreover, the two triangles $\T(A,B,A')$ and $\T(A',C,A)$ are isosceles and up to doing this procedure twice they are both acute or right.

Note that the last inequality above is strict as soon as the reflected torsion function associated to $\T(A,B,C)$ and restricted to the triangles $\T (A',C,A)$ and $\T (A',B,A)$ is not exactly the torsion function of the respective triangles.

We now show that Problem~\eqref{d70} has a solution among acute or right triangles that are not degenerate.

It is easy to notice that
\begin{equation}\label{d72}
 \lim_{r \to 0^+} \frac{\Tcal(\Teq,r)}{|\Teq\sm (V_{\Teq} \oplus B_r)|} =+\infty.
 \end{equation}
The precise asymptotic including this assertion is provided in Proposition \ref{prop:Convergence}.

Now fix $r >0$. A maximizing sequence for \eqref{d70} exists, call it $(\T_n)_n$, $\T_n =A_nB_nC_n$, such that $|\T_n\sm (V_{\T_n} \oplus B_r)| >0$. Since for every $n$, $\T_n\subset\ov {B(0,1)}$, we can assume that $A_n\to \bar A$, $B_n\to \bar B$ and $C_n\to \bar C$. All three points can not coincide, otherwise we contradict $|\T_n\sm (V_{\T_n} \oplus B_r)| >0$. If all pairwise distance between the points is larger than $2r$, then $\T=\bar A\bar B\bar C$ is the desired solution, since both torsion and area are continuous for this convergence. If not, this means that starting from some index, two points, say $B_n$ and $C_n$ have a distance smaller than $2r$. 

Consider $\T$ an acute or right triangle, of vertices $A,B,C$ such that the side $[B,C]$ has length smaller than $2r$. In this case, the two closed discs of radius $r$ centered in $B,C$ overlap. Assume the segment $[A,B]$ is aligned on the vertical axis $Oy$ of a system of coordinates. 
Then
$$-\frac 12\Tcal(\T,r)=\min \{\frac 12 \int _{\T\sm (V_\T \oplus B_r)}|\nabla u|^2 -
 \int _{\T\sm (V_\T \oplus B_r)}u : u \in H^1(\T\sm (V_\T \oplus B_r)) , u=0 \mbox{ on }V_\T \oplus \ov{B_r}\}$$
 $$\ge \min \{\frac 12 \int _{\T\sm (V_\T \oplus B_r)}u_y^2 -
 \int _{\T\sm (V_\T \oplus B_r)}u : u \in H^1(\T\sm (V_\T \oplus B_r)) , u=0 \mbox{ on }V_\T \oplus \ov{B_r}\}$$
 $$\ge \min \{\frac 12 \int _{\T\sm (V_\T \oplus B_r)}u_y^2 -
 \int _{\T\sm (V_\T \oplus B_r)}u : u \in H^1(\T\sm (V_\T \oplus B_r)) , u=0 \mbox{ on }B_r(B) \cup B_r(C)\}$$
$$= -\frac 12\int_I \frac{l(x)^3}{3} dx,$$
where $I$ is the projection of the segment $B_nC_n$ on the horizontal axis and $l(x)$ is the Hausdorff measure of 
$\T\sm (V_\T \oplus B_r)$ in the vertical line passing through $(x,0)$. This relies on the computation of the one dimensional torsion energy of an interval of length $l$ with Dirichlet condition at one end and Neumann on the other. 
Since $l(x) \le 2$ we get
$$-\frac 12\int_I \frac{l(x)^3}{3} dx \ge - \frac 43\int_I l(x)= - \frac 43 |\T\sm (V_\T \oplus B_r)|,$$
implying that
$$\frac{\Tcal(\T,r)}{|\T\sm (V_{\T} \oplus B_r)|} \le \frac 83.$$ .

The sequence $(\T_n)_n$ being maximizing, for $n$ large enough, 
$$\frac{\Tcal(\Teq,r)}{|\T\sm (V_{\Teq} \oplus B_r)|} \le \frac{\Tcal(\T_n,r)}{|\T_n\sm (V_{\T_n} \oplus B_r)|},$$
and therefore the previous computation imply, 
$$\frac{\Tcal(\Teq,r)}{|\T\sm (V_{\Teq} \oplus B_r)|} \le \frac 83.$$
Clearly, for small $r$ this is not optimal, being in contradiction to \eqref{d72}

We have then proved that there exists a solution to Problem~\eqref{d70} for $r<r^*$ small enough which is a non degenerate acute or right triangle, we call it $\Topt$, let us prove now that it has circumradius $\Rcal(\Topt) = 1$.

Assume $\Rcal(\Topt) <1$ and fix $t>1$ such that $\Rcal(t\Topt) = 1$.

We know $\Tcal(t\Topt, tr) = t^4 \Tcal(\Topt, r)$, so that since $\Topt$ is non degenerate,

$$ \frac{\Tcal(\Topt,r)}{|\Topt|-\frac{\pi r^2}2} = \frac{1}{t^4} \frac{\Tcal(t\Topt,tr)}{|\Topt|-\frac{\pi r^2}{2}} \le \frac{1}{t^4} \frac{\Tcal(t\Topt,r)}{|\Topt|-\frac{\pi r^2}{2}}. $$
The last inequality is a consequence of the monotonicity of the mixed torsional rigidity for the inclusion of the Dirichlet region.

If it holds,
$$ \frac{1}{t^4} \frac{1}{|\Topt|-\frac{\pi r^2}{2}} < \frac{1}{|t\Topt|-\frac{\pi r^2}{2}}, $$
then $t\Topt$ is a strictly better competitor than $\Topt$ and we have a contradiction.

Otherwise, assume 
$$ \frac{1}{t^4} \frac{1}{|\Topt|-\frac{\pi r^2}{2}} \ge \frac{1}{|t\Topt|-\frac{\pi r^2}{2}}, $$
which rewrites 
$$ |\Topt| \le (1 + t^{-2}) \frac{\pi r^2}2, $$
which implies, since $t>1$, $|\Topt| \le \pi r^2$.

Since $|\Topt|$ is non degenerate and right or acute, it has circumradius $\Rcal(\Topt) > r$ and the height associated to each edge of the triangle must be greater than $r$. This implies that if any edge of $\Topt$ has length greater than $7r$, $|\Topt| \ge 3.5 r^2 > \pi r^2$. Consequently, all the edges of $\Topt$ have length smaller than $7r$.

We estimate $\Tcal(\Topt,r)$ by a 1D torsional rigidity of radial type originating from the vertex $A$,

$$ \Tcal(\Topt, r) = \max \Big\{\int_{\Topt} \big[2u-|\nabla u|^2\big] \Btq u\in H^1(\Topt),\ u = 0, \mbox{ on } \Topt\cap[V_{\Topt}\oplus \ov{B_r}]\Big\} $$
$$ \le \max \Big\{\int_0^\alpha \int_r^{\rho(\theta)} \big[2\rho\, u - \rho\,u_\rho^2 \big]\d\rho\d\theta \Btq u\in H^1(\Topt),\ u=0 \mbox{ on } \ov{B_r(A)} \Big\}$$
where $\alpha$ denotes the angle in $A$. Solving the 1D, problem on the segment $[r,\rho(\theta)]$
$$ S_{1D} = \max_{f, f(r) = 0} \int_{r}^{\rho(\theta)} 2s\,f(s)- s\,(f'(s))^2 \d s,$$ 
we find that the maximum is obtained for the function
$$ f(s) = \frac{\rho(\theta)^2}{2}\ln(\frac{s}{r}) + \frac{r^2-s^2}{4} \le \frac{(7r)^2}2 \ln(7), $$
and the energy satisfies 
$$ S_{1D} \le (7r)^2 \ln(7) \int_r^{\rho(\theta)} s\d s .$$ 
This implies
$$ \Tcal(\Topt,r) \le (7r)^2 \ln(7) (|\Topt|-\frac{\pi r^2}2), $$
so that, 
$$ \frac{\Tcal(\Topt,r)}{|\Topt| - \frac{\pi r^2}2} \le 49 \ln(7)\, r^2, $$
which cannot be optimal for $r < r^*$ small since it contradicts \eqref{d72}. 
\end{proof}
 
Consequently, from now on, we consider $r<r^*$ with $r^*$ small enough such that the solutions of Problem~\eqref{d70}, are non degenerate acute or right triangles of circumradius $1$.
 
Let $\T$ be an acute triangle.We make the following manipulation: we divide the triangle $\T$ into three isosceles triangles of angles $\theta_1,\theta_2,\theta_3$ at the center of the circle. Each one is divided into two congruent right triangles (Figure~\ref{fig:symmetrization}, right) of angles $(\alpha_i , \frac{\theta_i}2, \frac\pi2)$, where $\alpha_i=\frac\pi2- \frac{\theta_i}2$. Formally, if $\T$ is an isosceles right triangle, then $\alpha_1 = 0, \alpha_2 = \alpha_3 =\pi/4$. 

For $0<\alpha<\pi/2$, let us denote by $\T^\alpha$ the right triangle of vertices $O = (0,0)$, $A=(\cos(\alpha),0)$, $B=(\cos(\alpha),\sin(\alpha))$ (Figure~\ref{fig:righttriangleNEW}) and angles $ \hat O =\alpha, \hat A =\frac\pi2, \hat B= \frac \theta2$. We denote by $w_{r,\alpha}$ the mixed torsion function associated to this triangle:
$$
\left\{\begin{array}{ll}
    -\Delta w_{r,\alpha} = 1 &\text{ in } \T^\alpha\setminus \ov B_r,\\
    w_{r,\alpha} = 0 & \text{ in }B_r,\\
    \partial_n w_{r,\alpha} = 0 &\text{ on } \partial \T^\alpha.
\end{array}\right.
$$
Its mixed torsional rigidity equals
$$ J_r(\alpha) = \int_{\T^\alpha} w_{r,\alpha}\,dx, $$
By continuous extension, we define $J_r(0) = J_r(\frac \pi2)=0 $. 

\begin{figure}
\centering
\includegraphics[scale = 0.7]{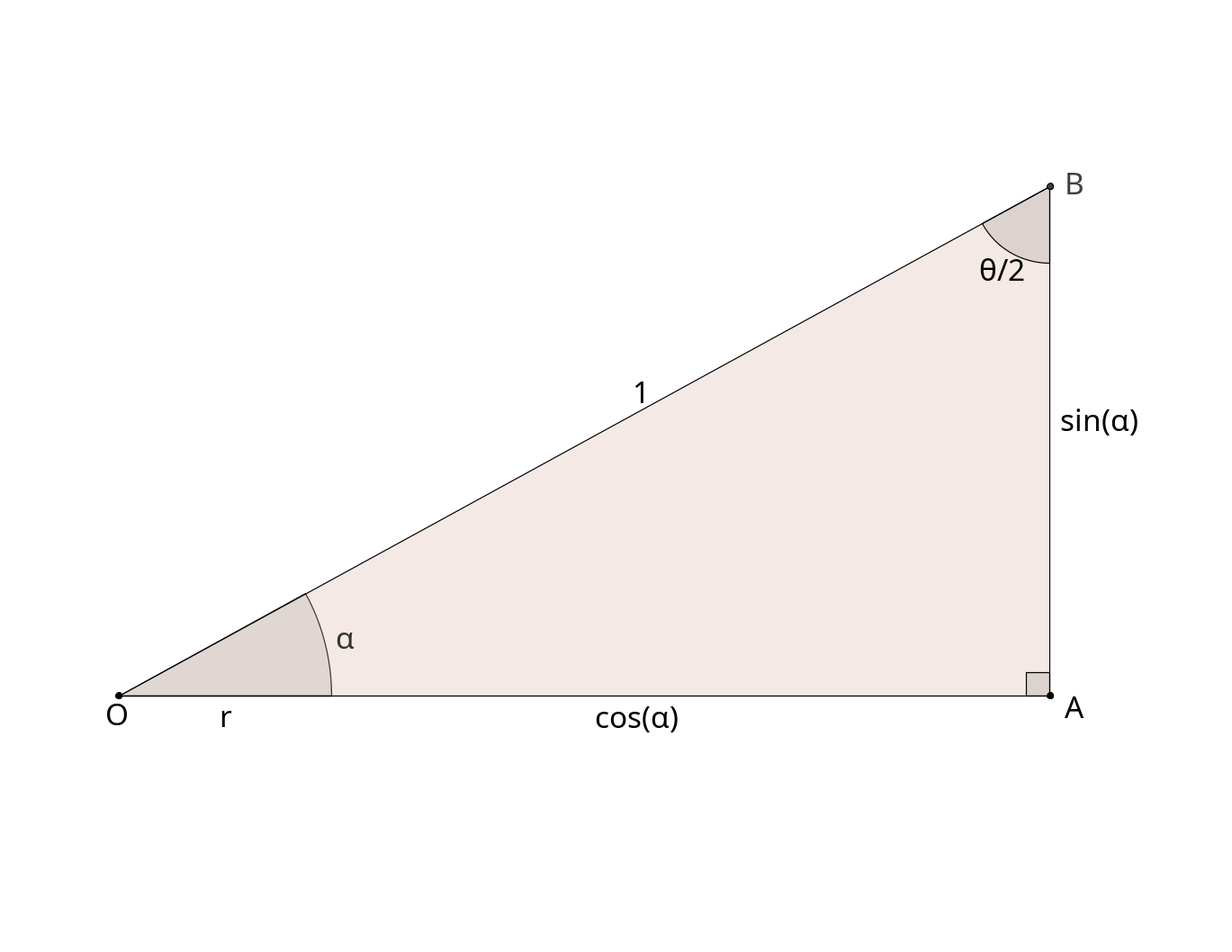}
\caption{The triangle $\T^\alpha$ \label{fig:righttriangleNEW}}
\end{figure}

We have,
$$\frac{\Tcal(\T,r)}{|\T\sm [V_\T\oplus \ov{B_r}]|}\le\frac{J_r(\alpha_1) + J_r(\alpha_2) + J_r(\alpha_3)}{|\T^{\alpha_1}\setminus B_r | + |\T^{\alpha_2}\setminus B_r | + |\T^{\alpha_3}\setminus B_r |}.$$
Equality occurs when the triangle $\T$ is equilateral, otherwise the inequality is strict.

We introduce the functional 

\be\label{eq:defFr}
F_r(\alpha_1, \alpha_2) = \frac{J_r(\alpha_1) + J_r(\alpha_2) + J_r(\pi/2-(\alpha_1+\alpha_2))}{|\T^{\alpha_1}\setminus B_r | + |\T^{\alpha_2}\setminus B_r | + |\T^{\pi/2- (\alpha_1 + \alpha_2)}\setminus B_r |},
\ee
that we define for $(\alpha_1,\alpha_2)\in\S$ the set of admissible couples,
\be \label{eq:defS}
\S = \{(\alpha_1,\alpha_2) \in [0,\pi/4]^2~|~ \alpha_1+\alpha_2 > 0\}.
\ee

For a fixed $r>0$ we introduce the set of non degenerate couples $\S_{ND,r}=\{(\alpha_1,\alpha_2) \in [0,\pi/4]^2~|~ \alpha_1+\alpha_2 \ge \eta_r\}$ with $\sin(\eta_r) =r$. These couples correspond to the non degenerate triangles $\T$. Note that any given couple $(\alpha_1,\alpha_2)\in S$ becomes non degenerate when $r$ gets small.

In the case where $(\alpha_1, \alpha_2)$ is non degenerate, $F_r$ can be written, 
$$ F_r(\alpha_1, \alpha_2) = 4 \frac{J_r(\alpha_1)+J_r(\alpha_2)+J_r(\pi/2-(\alpha_1+\alpha_2))}{\sin(2\alpha_1)+\sin(2\alpha_2)+\sin(2(\alpha_1+\alpha_2)) - \pi r^2}.$$

Our goal is to show that for $r<r^*$ sufficiently small, $F_r$ is maximized by, and only by, the couple $(\pi/6, \pi/6)$ which corresponds to the equilateral triangle. We split the proof in a series of intermediate results, stated in Propositions \ref{prop:Convergence}, \ref{prop:maximalityF0}, and \ref{prop:MaxFr} which, roughly speaking, bring the analysis of the maximization of $F_r$ for small $r$ to the analysis of a limit problem obtained as the leading asymptotic term when $r \to 0$. This problem is of purely geometrical type and can be solved explicitly. We introduce the function 
\be\label{eq:defJ0NEW}
J_0: [0,\pi/2] \to \R, \;\; J_0(\alpha):=\frac{\sin^2(2\alpha)}{16\alpha}, \;\; J_0(0)=J_0(\frac \pi2)=0.
\ee
 
\begin{prop} \label{prop:Convergence}
The following assertions hold :
\begin{enumerate}
\item There exists a positive constant $\overline C$ such that for any $r <r^*$ sufficiently small, 
$$ \forall \alpha\in(0,\pi/2),\qquad \left|\frac{1}{|\ln(r)|}\frac{J_r(\alpha)}{|\T^\alpha|} - \frac{J_0(\alpha)}{|\T^\alpha|}\right|\le\frac{\overline C}{|\ln(r)|}. $$

\item For any $0<\alpha_m<\alpha_M<\pi/2$, the following limit hold uniformly on $[\alpha_m,\alpha_M]$, 
$$\lim_{r\to0} \left|\frac{J_r''(\alpha)}{|\ln(r)|}- J_0''(\alpha)\right|=0$$
\end{enumerate}
\end{prop}

Proposition~\ref{prop:Convergence} implies the following limit for the functional $F_r$, for any $(\alpha_1,\alpha_2)\in \S$,
\be\label{eq:limitFr}
\lim_{r\to 0} \frac{F_r(\alpha_1, \alpha_2)}{|\ln(r)|} = \frac{\frac{\sin^2(2\alpha_1)}{\alpha_1}+\frac{\sin^2(2\alpha_2)}{\alpha_2}+\frac{\sin^2(2(\alpha_1+\alpha_2))}{\pi/2-(\alpha_1+\alpha_2)} }{4(\sin(2\alpha_1)+\sin(2\alpha_2)+\sin(2(\alpha_1+\alpha_2)))} := F_0(\alpha_1, \alpha_2).
\ee

We study the functional $F_0(\alpha_1,\alpha_2)$. 

\begin{prop} \label{prop:maximalityF0}
For any couple of real numbers $(\alpha_1,\alpha_2) \in \S$, we obtain
$$
F_0(\alpha_1, \alpha_2) \le \frac{3\sqrt 3}{4\pi} :=E_0
$$
with equality if and only if $\alpha_1 = \alpha_2 = \pi/6$.
\end{prop}

Concerning $F_r$, we denote by $E_r$ its supremum on $\S$ and we show the following result.

\begin{prop}\label{prop:MaxFr}
For $0<r<r^*$, the functional $F_r$ admits maximizers that are non degenerate and in addition, it holds: 
$$\lim_{r\to 0}\frac{E_r}{|\ln(r)|}=E_0.$$
and, for any sequence of maximizers, $(\alpha_{1,r},\alpha_{2,r})_r$,
$$\lim_{r\to0} \alpha_{i,r} = \frac{\pi}{6}\qquad\text{for }i=1,2.$$
\end{prop}

\smallskip
\noindent{\bf Proof of Theorem~\ref{TH2}.}
Now, let us assume that one can find a sequence $(r_n)$ going to $0$ such that for all $n$ there exists a non degenerate couple $\bar\alpha_{i,r_n}$, $i=1,2$, that are not all equal to $\frac{\pi}{6}$ and maximize $F_{r_n}$. We denote by $\bar\alpha_{3,r_n}$ the quantity $\bar\alpha_{3,r_n} = \pi/2 -(\bar\alpha_{1,r_n}+\bar\alpha_{2,r_n})$. From Proposition~\ref{prop:MaxFr}, we know that the three sequence $\bar \alpha_{i,r_n} \to \pi/6$ when $n\to +\infty$, and, without loss of generality, we may assume $\pi/6 -\eps < \bar\alpha_{1,r_n}\le \bar\alpha_{2,r_n} \le \bar\alpha_{3,r_n}< \pi/6 + \eps$ for some $\eps>0$. From the maximality and non degeneracy, we get using Lagrange multipliers that there exists a constant $l_{r_n}$ such that
$${J_{r_n}}'(\bar\alpha_{i,r_n}) - \frac{E_{r_n}}2\ \cos(2\bar\alpha_{i,r_n})=l_{r_n}\qquad\text{for every }i.$$
Then there exists $\xi_n$, such that, $\bar\alpha_{1,r_n} < \xi_n <\bar\alpha_{3,r_n}$ and
\be\label{eq:xin}
 {J_{r_n}}''(\xi_n) + E_{r_n} \ \sin(2\xi_n) = 0. 
\ee
Since $\bar \alpha_{i,r_n} \to \pi/6$ for $i = 1,3$, we deduce that $\xi_n\to\frac{\pi}{6}$. Additionally, the second assertion of Proposition~\ref{prop:Convergence} implies that $J''_{r_n}/|\ln(r_n)|$ converges uniformly on $[\pi/6 - \eps, \pi/6 + \eps]$ to $J_0''$. Therefore, dividing by $|\ln(r_n)|$ we can pass to the limit in \eqref{eq:xin} and obtain
$$J_0''\left(\frac{\pi}{6}\right) + \frac{\sqrt3}2 E_0 = 0.$$
Yet, a direct computation gives 
\be\label{eq:contradictionNEW}
J_0''\left(\frac{\pi}{6}\right) + \frac{\sqrt3}2 E_0 = \frac{3(54 - 12\sqrt 3 \pi - \pi^2)}{8\pi^3} <0.
\ee
We then have a contradiction and we deduce that there exists a $r^*>0$ such that for $r<r^*$, the only maximizer of $F_r$ is the equilateral triangle. This proves Theorem~\ref{TH2}.
\end{proof}

\section{Proof of Proposition~\ref{prop:Convergence}}\label{sec:proofPropConv}

To prove Proposition~\ref{prop:Convergence}, we are going to compare the torsion function $w_{r,\alpha}$ to the explicit radial function $v_{r,\alpha}$, defined as (the restriction to $\T^\alpha\setminus B(O,r)$ of) 
$$ v_{r,\alpha}(x) = \frac{\sin(2\alpha)}{4\alpha} \ln\left(\frac{|x|}r\right) + \frac{r^2-|x|^2}4.$$
It satisfies the following equation, 
\be
\left\{\begin{array}{ll}
-\Delta v_{r,\alpha} =1 & \text{ in } \T^\alpha\setminus B(O,r),\\
v_{r,\alpha} =0 &\text{ on } \partial B(O,r),\\
\partial_n v_{r,\alpha} = 0 & \text{ on } [O,A], [O,B],\\
\partial_n v_{r,\alpha} = g^\alpha & \text{ on } [\bar A,B]
\end{array}\right.
\ee
where $g^\alpha$ has for expression, for $y\in[0,\sin(\alpha)]$,
$$ g^\alpha(y) = \frac{\sin(2\alpha)}{4\alpha} \frac{\cos(\alpha)}{\cos(\alpha)^2+y^2} - \frac{\cos(\alpha)}2.$$
Let us denote by $H_r(\alpha)$ the associated (pseudo) torsional rigidity, 
$$ H_r(\alpha) = \int_{\T^\alpha} v_{r,\alpha}.$$

 Here, by $\bar A$ we denote, when it exists, the point of intersection between the circle $\partial B(O,r)$ and the segment $[A,B]$, if it doesn't exists we still write $\bar A = A$.

\begin{rem}
Intuitively, the function $v_{r,\alpha}$ is a kind of torsion function of the sector of radius $1$ and angle $\alpha$ with Dirichlet boundary conditions on $B(O,r)$ and Neumann boundary conditions on the two radii, rescaled to have the same limit as $w_{r,\alpha}$. It is designed to compensate the $\log$ singularity of $w_{r,\alpha}$ near $B(O,r)$ when $r\to 0$.
\end{rem}
In the following we assume that the functions $w_{r,\alpha}$ and $v_{r,\alpha}$ are defined on the whole triangle $\T^\alpha$, extended by $0$ on $B(O,r)\cap \T^\alpha$.

\subsection{The limits of \(H_r\) and its derivatives}

To begin, we show the following lemma:

\begin{lem}\label{lem:limitH}
The following assertions hold :
\begin{enumerate}
    \item There exists a positive constant $\overline C$ such that for any $r<r^*$ sufficiently small, 
$$
\forall\alpha\in(0,\pi/2),\qquad\left|\frac{H_r(\alpha)}{|\ln(r)|}-J_0(\alpha)\right|\le\frac{\overline C |\T^\alpha|}{|\ln(r)|}.
$$
    \item For any $0<\alpha_m<\alpha_M<\pi/2$, There exists two constants $C(\alpha_m,\alpha_M)$ and $\bar r(\alpha_m,\alpha_M)$ depending on $\alpha_m$ and $\alpha_M$ such that for any $r<\bar r(\alpha_m,\alpha_M)$, 
$$
\forall \alpha\in[\alpha_m,\alpha_M],\qquad \left|\frac{H''_r(\alpha)}{|\ln(r)|} - J''_0(\alpha)\right| \le C(\alpha_m,\alpha_M).
$$
\end{enumerate}
\end{lem}
\begin{proof}
To prove this lemma, since $H_r$ is explicit, we compute directly $H_r$ and its derivatives.

Let $0<r<r^*$ with $r^*<1$, small and $\alpha\in(0,\pi/2)$. Since the expression of $v_{r,\alpha}$ is elementary, a direct computation in polar coordinates yields
$$
\begin{aligned} 
H_r(\alpha) = &|\T^\alpha\setminus B(O,r)| \Big(|\ln(r)| \frac{J_0(\alpha)}{|\T^\alpha|} - \frac{r^2}4 + \frac{\sin(2\alpha)}{8\alpha}\Big)\\
&+ \frac{\sin(2\alpha)}{4\alpha}\int_{\alpha_0}^\alpha \Big(\frac{R(\phi)^2}2 \ln(R(\phi))-\frac{r^2}2\ln(r)\Big)\d\phi -\frac1{16}\int_{\alpha_0}^\alpha \big(R(\phi)^4 - r^4)\d\phi,
\end{aligned}
$$
where $R(\phi) = \frac{cos(\alpha)}{cos(\phi)}$ and
$$ \alpha_0 = \left\{
\begin{array}{ll}
0 &\text{if } \cos(\alpha)>r\\
\arccos(\frac{\cos(\alpha)}{r}) &\text{otherwise}.
\end{array}
\right. 
$$
Noting that the function $x\to x^2/2 \ln(x)$ stays bounded on $(0,1]$, we can easily estimate the two remaining integral terms and obtain for some $C>0$,
$$
\begin{aligned}
&\Big |\int_{\alpha_0}^\alpha \Big(\frac{R(\phi)^2}2 \ln(R(\phi))-\frac{r^2}2\ln(r)\Big)\d\phi\Big | \le C\, (\alpha-\alpha_0) \le C\,\alpha,\\
&\Big |\int_{\alpha_0}^\alpha \big(R(\phi)^4 - r^4)\d\phi\Big | \le 4 |\T^\alpha\setminus B(O,r)|.
\end{aligned}
$$
Now, notice that $(1-r)^2|\T^\alpha|\le|\T^\alpha\setminus B(O,r)| \le |\T^\alpha|$, so for a constant $\ov C$ independent of $\alpha$ and $r$ we get
$$ \Big |H_r(\alpha) - |\ln(r)| J_0(\alpha)\Big| \le \overline C \,|\T^\alpha|. $$
This proves assertion ($1$).

Now, take two constants, $0<\alpha_m<\alpha_M<\pi/2$ and assume that $r<\cos(\alpha_M)$. In this setting, for any $\alpha\in(\alpha_m,\alpha_M)$, we have that $\alpha_0 = 0$ and $H_r(\alpha)$ writes as 
$$\begin{aligned} 
H_r(\alpha) =& |\ln(r)| J_0(\alpha) +\frac{\sin(2\alpha)}{16} \Big(r-\frac16\Big) - \frac{\sin(2\alpha)}{32\alpha}r^2 + \frac{3\alpha}{16}r^4 +\frac{\sin(2\alpha)^2}{32\alpha^2}\\
&+ \frac{\sin(\alpha)\cos(\alpha)^3}{12} - \frac{\sin(2\alpha)^2}{16\alpha}.
\end{aligned}$$
Here we used the two elementary computations, 
$$ \int_0^\alpha \frac{1}{\cos(\phi)^2} \ln(\cos(\phi))\d\phi = \tan(\alpha)\ln(\cos(\alpha)) +\tan(\alpha) -\alpha, $$
$$\int_0^\alpha\frac{1}{\cos(\phi)^4}\d\phi = \frac{\tan(\alpha)}{3} \Big(\frac1{\cos(\alpha)^2} + 2\Big).$$
We can then compute,
$$ \begin{aligned}
H_r''(\alpha) =&|\ln(r)| J_0''(\alpha) - \frac{\sin(2\alpha)}4 \Big(r-\frac16\Big) \\
& +\frac{\sin(2\alpha)}{8\alpha} r^2 + \frac{\cos(2\alpha)}{8\alpha^2} r^2 - \frac{\sin(2\alpha)}{16\alpha^3}r^2\\
& +\frac{\cos(4\alpha)}{4\alpha^2} - \frac{\sin(4\alpha)}{4\alpha^3} + \frac{3\sin(2\alpha)^2}{16\alpha^4}\\
& - \frac{\sin(\alpha)\cos(\alpha)^3}3 - \frac{\sin(4\alpha)}8\\
& - \frac{\cos(4\alpha)}{2\alpha} + \frac{\sin(4\alpha)}{4\alpha^2} - \frac{\sin(2\alpha)^2}{8\alpha^3}.
\end{aligned} $$
Then for a constant depending only on the choice of $\alpha_m,\alpha_M$, denoted $C(\alpha_m,\alpha_M)$, we have for any $\alpha\in(\alpha_m,\alpha_M)$,
$$ |H''_r(\alpha) - |\ln(r)|J_0''(\alpha)| \le C(\alpha_m,\alpha_M).$$
This proves assertion ($2$).
\end{proof}

\subsection{The zero order convergence}
Let us introduce the function $\psi_{r,\alpha} = w_{r,\alpha} -v_{r,\alpha}$. 
As a first step we prove that $\psi_{r,\alpha}$ satisfies the following.

\begin{lem}\label{lem:PoincarePsi}
There exists constants $\overline P, r^*>0$ such that for any $0<r<r^*$ and for any $\alpha\in (0,\pi/2)$ we have
$$\int_{\T^\alpha} \psi_{r,\alpha}^2 \le \overline P \int_{\T^\alpha} |\nabla \psi_{r,\alpha}|^2.$$
\end{lem}
 
\begin{proof}
{\it In the following, we denote by $C$ a positive constant which may increase from line to line.}
Fix $0<r<r^*$ small and $\alpha\in(0,\pi/2)$. First, assume $\cos(\alpha) \le r$. We make a reflection of the triangle $\T^\alpha$ along its vertical edge $[A,B]$. We obtain a triangle $\tilde \T^\alpha$ of vertices $O,B,\tilde O$ on which we define a function $\tilde \psi_{r,\alpha}$ by reflection of $\psi_{r,\alpha}$. It satisfies naturally
$$ \frac{\int_{\tilde \T^\alpha} |\nabla \tilde \psi_{r,\alpha}|^2}{\int_{\tilde\T^\alpha} \tilde\psi_{r,\alpha}^2} = \frac{\int_{ \T^\alpha} |\nabla \psi_{r,\alpha}|^2}{\int_{\T^\alpha} \psi_{r,\alpha}^2}, $$
and additionally, it satisfies Neumann boundary conditions on the the sides $[O,B]$ and $[B,\tilde O]$ and Dirichlet boundary condition of the side $[O,\tilde O]$ of $\tilde \T^\alpha$. Extending the Dirichlet region, we have that $\tilde\psi_{r,\alpha}$ is an $H^1$ function of the sector $S(\pi-2\alpha)$ of angle $\pi-2\alpha$ and radius $1$ which vanish on the arc circle and satisfies Neumann boundary conditions on the two radii. Since the first Dirichlet eigenfunction of the ball of radius $1$ is radially symmetric, it is also eigenfunction on each sector, so that
$$ \frac{\int_{\tilde \T^\alpha} |\nabla \tilde \psi_{r,\alpha}|^2}{\int_{\tilde\T^\alpha} \tilde\psi_{r,\alpha}^2} \ge \lambda_1(B_1).$$
And therefore this also holds for the function $\psi_{r,\alpha}$,
$$ \int_{\T^\alpha} \psi_{r,\alpha}^2 \le \frac1{\lambda_1(B_1)} \int_{\T^\alpha} |\nabla\psi_{r,\alpha}|^2.$$
\smallskip

Now, consider that $\cos(\alpha) > r$. In this setting, $\bar A=A$.
From the equations satisfied by the two functions, $w_{r,\alpha}$ and $v_{r,\alpha}$ we find that $\psi_{r,\alpha}$ satisfies for any $u\in H^1(\T^\alpha)$, $u = 0$ on $B(O,r)\cap \T^\alpha$, 
$$ \int_{\T^\alpha} \nabla \psi_{r,\alpha} \cdot \nabla u = -\int_{[A,B]} g^\alpha \, u. $$

In the spirit of Cioranescu and Murat, we consider the function $z_{r,\alpha}$ 
$$
z_{r,\alpha} = \left\{\begin{array}{ll}
0 &\text{ in } B(O,r),\\
1 &\text{ in } \R^2\setminus B(O,\cos\alpha),\\
\frac{\ln(|x|/r)}{\ln(\cos(\alpha)/r)} &\text{ in }B(O,\cos\alpha) \setminus B(O,r),
\end{array}\right.
$$
and test the equation with $z_{r,\alpha}|_{\T^\alpha\sm \ov B(O,r)}$.
Since $\int_{[A,B]} g^\alpha = 0$ and $z_{r,\alpha}|_{[A,B]}=1$, we find
$$ \int_{\T^\alpha}\nabla\psi_{r,\alpha} \cdot \nabla z_{r,\alpha} = -\int_{[A,B]} g^\alpha\,z_{r,\alpha} = 0. $$

Now, relying on the expression of $z_{r,\alpha}$ and integrating by parts on $\T^\alpha \cap (B(O,\cos\alpha) \setminus B(O,r))$, we deduce 
$$ \int_{\T^\alpha} \nabla\psi_{r,\alpha} \cdot \nabla z_{r,\alpha} = \frac{1}{\cos(\alpha)\ln(\cos(\alpha)/r)} \int_{\partial B(O,\cos(\alpha))\cap\T^\alpha} \psi_{r,\alpha}, $$
and therefore,
$$ \int_{\partial B(O,\cos(\alpha))\cap\T^\alpha} \psi_{r,\alpha} =0.$$

Let's first focus on the angular sector, $\T^\alpha\cap B(O,\cos(\alpha))$. The function $\psi_{r,\alpha}$ has zero mean value on the arc of circle $\partial B(O,\cos(\alpha))$. The first nontrivial eigenvalue of the angular sector of radius $1$ with with Steklov condition on the arc and Neumann conditions on the radii equals $\pi/\alpha$ and scales as $t^{-1}$. Therefore, $\psi_{r,\alpha}$ satisfies
$$\frac{\pi}{\alpha\cos\alpha} \int_{\partial B(O,\cos \alpha)\cap \T^\alpha} \psi_{r,\alpha}^2 \le \int_{B(O,\cos \alpha)\cap \T^\alpha} |\nabla \psi_{r,\alpha}|^2.$$
Denoting $\lambda_1(B_1,1)$ the first Robin eigenvalue of the disc of radius $1$ with Robin parameter $1$, we get after rescaling $\lb_1(B_{\cos \alpha}, \frac{1}{\cos \alpha})= \frac 1{\cos^2\alpha} \lambda_1(B_1,1)$. Since the first Robin eigenfunction on the disc is radially symmetric, it also eigenfunction on the sector, so that the eigenvalues are the same and
$$\frac{\lambda_1(B_1,1)}{\cos^2\alpha} \int_{ B(O,\cos \alpha )\cap\T^\alpha} \psi_{r,\alpha}^2 \le\int_{ B(O,\cos\alpha )\cap\T^\alpha} |\nabla\psi_{r,\alpha}|^2 + \frac1{\cos\alpha } \int_{\partial B(O,\cos\alpha)\cap\T^\alpha} \psi_{r,\alpha}^2.$$
Combining the two, we deduce that,
$$ \int_{ B(O,\cos\alpha)\cap \T^\alpha} \psi_{r,\alpha}^2 \le C \int_{B(O,\cos\alpha)\cap \T^\alpha} |\nabla \psi_{r,\alpha}|^2 \le C \int_{\T^\alpha} |\nabla \psi_{r,\alpha}|^2.$$

We claim that there exists a constant $C>0$ independent on $\alpha\in (0, \frac \pi2)$ such that 
\begin{equation}\label{d14}
C\int_{\T^\alpha\sm B(O,\cos\alpha)}\psi_{r,\alpha}^2 \le \int_{\T^\alpha\sm B(O,\cos\alpha)} |\nabla\psi_{r,\alpha}|^2+\int_{\partial B(O,\cos\alpha)\cap\T^\alpha} \psi_{r,\alpha}^2.
\end{equation}
Indeed, the set $ \T^\alpha\sm B(O,\cos\alpha)$ is a cone with vertex $B$, aperture $\frac\pi2-\alpha$ of area 

$$a_\alpha=\frac{\sin \alpha \cos \alpha}{2}- \alpha\frac{\cos^2\alpha}2.$$

The first mixed Robin-Neumann eigenvalue in this cone, with Robin condition on $\partial B(O, \cos\alpha) \cap \T^\alpha$ with parameter $1$ and Neumann condition on the radii, is not smaller than the similar eigenvalue of the sector of aperture $\frac \pi2 -\alpha$ and radius $r_\alpha$ where $r_\alpha$ is chosen to keep the same area, by 

$$\frac{\frac\pi2-\alpha}{2}r^2_\alpha= a_\alpha.$$
This assertion can be proved following step by step the strategy of Bossel-Daners \cite{Da06} based on the H-function, the radiality of the Robin eigenfunction on the ball and the isoperimetric inequality. In a sector, the Robin-Neumann first eigenfunction is the restriction of the Robin eigenfunction of the disc and the relative isoperimetric sets in a cone are the sectors.

The sector eigenvalue equals $\lb_1(B_{r_\alpha}, 1)$. Note that $r_\alpha$ remains bounded from above since $ \cos \alpha \sim \frac\pi2-\alpha$ when $\alpha \to \frac \pi2^-$. As a consequence, $\lb_1(B_{r_\alpha}, 1)$ stays bounded away from $0$, uniformly on $\alpha \in (0, \frac\pi2)$, thus leading to the validity of \eqref{d14}.

Finally, putting together all the previous inequalities, we get
$$ \int_{\T^\alpha\setminus B(O,\cos\alpha )} \psi_{r,\alpha}^2 \le C \int_{\T^\alpha} |\nabla\psi_{r,\alpha}|^2 $$
with a constant $C$ independent on $\alpha \in (0, \frac \pi2)$. 
\end{proof}

Below, we provide a series of useful bounds satisfied by the functions $w_{r,\alpha}$ and $v_{r,\alpha}$.
\begin{lem}\label{lem:zeroOrderBounds}
There exists a positive constant $\overline C$ such that for any $r<r^*$ sufficiently small, and any $\alpha \in (0,\pi/2)$, the following bounds hold:

\begin{enumerate}
\item $$ |J_r(\alpha) -H_r(\alpha)| \le \overline C \,|\T^\alpha|^\frac 32s,$$
\item $$ \|\nabla w_{r,\alpha} -\nabla v_{r,\alpha}\|_{L^2(\T^\alpha)}\le \overline C\,|\T^\alpha|^\frac 12,$$
\item $$ \frac{\|\nabla v_{r,\alpha}\|_{L^2(\T^\alpha)}}{\sqrt{|\ln(r)|}}\le \overline C\, |\T^\alpha| ^\frac 12 ,$$
\item $$ \frac{\|\nabla w_{r,\alpha}\|_{L^2(\T^\alpha)}}{\sqrt{|\ln(r)|}}\le \overline C\, |\T^\alpha| ^\frac 12 .$$
\end{enumerate}
\end{lem}
 
\begin{proof}
{\it In the following, we denote by $C$ a positive constant which may increase from line to line.} 
Fix $0<r<r^*$, with $r^*$ small and $\alpha\in(0,\pi/2)$. For all $u\in H^1(\T^\alpha)$ which vanish on $B(O,r)\cap \T^\alpha$, the weak form of the equation satisfied by $\psi_{r,\alpha}$ gives
$$ \int_{\T^\alpha} \nabla \psi_{r,\alpha}\cdot \nabla u = - \int_{[\bar A,B]} g^\alpha\,u. $$

Inserting $u= \psi_{r,\alpha}$ and relying on Lemma~\ref{lem:PoincarePsi}, we get
$$\int_{\T^\alpha}|\nabla\psi_{r,\alpha}|^2 \le (1+\overline P) \,C_T(\alpha)\,\int_{[A,B]}{(g^\alpha)}^2,$$
where 
$$C_T(\alpha) = \sup\Big\{\frac{\int_{[A,B]}u^2}{\int_{\T^\alpha} u^2 +\int_{\T^\alpha}|\nabla u|^2} ~\Big|~ u\in H^1(\T^\alpha)\Big\},$$ 
denotes the square of the trace constant associated to the trace operator $H^1(\T^\alpha)\hookrightarrow L^2([A,B])$. This constant varies continuously on $(0,\alpha)$. Let us check its behavior when $\alpha\to 0$ and $\alpha\to\pi/2$. 

\smallskip
\noindent {\bf The case $ (\alpha \to 0^+)$.} When $\alpha$ is small, we reflect along the side $[O,B]$, obtaining a quadrilateral. The segment $[O,A]$ is reflected to $[O, A_1]$. Starting with $[O, A_1]$, we reflect successively $\lfloor \frac{\pi}{4\alpha}\rfloor$ times the quadrilateral with respect to $[ O, A_i]$. Finally, we obtain a domain $Q(\alpha)$ which, when $\alpha \to 0^+$, approaches the quarter of a disk. We estimate the trace on the segment $[A,B]$ estimating the $L^2$-norm of the extended function on all reflections of $[A,B]$ agains the full norm in the extended domain and get that 
$$C_T(\alpha) \le \overline C_T(Q(\alpha)).$$
Above, for a generic bounded, open, Lipschitz set $\O\subset\Rtwo$, we denote by $\overline C_T(\O)$ the square of the continuity constant of the trace operator $H^1(\O) \hookrightarrow L^2(\partial \O)$. Now, since $\alpha$ is small, the family of domains $(Q(\alpha))_\alpha$ is uniformly bounded and satisfy a uniform cone condition, so there is a universal constant $C^*$ such that 
$$ \overline C_T(Q(\alpha)) \le C^*. $$
In particular, for $\alpha$ small, the constant $C_T(\alpha)$ remains bounded. 

\smallskip
\noindent {\bf The case $ (\alpha \to \frac \pi2^-)$.} 
We use a similar reflection argument. We first reflect across $[B,O]$ and obtain $[B, A_1]$ for the reflection of $[B,A]$, and then successively across $[B, A_i]$ the quadrilateral obtained in this way, $\lfloor \frac{\pi}{4(\frac\pi2-\alpha)}\rfloor$ times, building the set $Q(\pi/2-\alpha)$, close to a quarter of a disc.
In the regime $\alpha$ close to $\pi/2$, we still have 
$$ \overline C_T(Q(\pi/2-\alpha)) \le C^*,$$
so that 
$$C_T(\alpha) \le \frac{C}{\pi/2-\alpha}\approx \frac{ C}{\cos\alpha}, \mbox{when}\; \alpha \to \frac \pi2 ^-.$$
The extra factor $\frac{1}{\pi/2-\alpha}$ comes from the fact that the trace on a single, fixed, segment $[A,B]$ is bounded by the full norm in $Q(\pi/2-\alpha)$, which means that on the initial triangle this factor pops up.

From this analysis, we deduce that we can bound for any $\alpha\in(0,\pi/2)$,
$$C_T(\alpha) \le \frac{ C}{\pi/2-\alpha}.$$

Additionally, relying on the elementary computation
\[
\int_0^{\sin\alpha} \frac{dy}{(\cos^2\alpha+y^2)^2}
=
\frac{\sin\alpha}{2\cos^2 \alpha(\cos^2\alpha+\sin^2\alpha)}
+\frac{1}{2\cos^3\alpha}\arctan\!\left(\frac{\sin \alpha}{\cos\alpha}\right),
\]
direct computation gives
$$ \int_{[A,B]} ({g^\alpha})^2\le C\,\cos\alpha. $$
Therefore, the product,
$$ C_T(\alpha) \int_{[A,B]} {(g^\alpha)}^2,$$
remains bounded for $\alpha \in(0,\pi/2)$.

Consequently, we obtain that for any $0<r<r^*$ and $\alpha\in(0,\pi/2)$, 
$$ \int_{\T^\alpha} |\nabla \psi_{r,\alpha}|^2 \le C. $$

From Cauchy-Schwartz and Lemma \ref{lem:PoincarePsi}, we obtain 
$$ \Big|\int_{\T^\alpha} \psi_{r,\alpha}\Big|^2 \le \overline P\, |\T^\alpha|\, \int_{\T^\alpha} |\nabla \psi_{r,\alpha}|^2 \le C\, |\T^\alpha|. $$ 

Now, relying on the equations satisfied by $\psi_{r,\alpha}$ and $v_{r,\alpha}$, we obtain 
$$\int_{\T^\alpha} |\nabla \psi_{r,\alpha}|^2 = \int_{\T^\alpha} \psi_{r,\alpha} + \int_{[A',B]} g^\alpha\, v_{r,\alpha}.$$

By direct computation, for a positive constant $C$ we get
$$\int_{[\bar A,B]} g^\alpha v_{r,\alpha} \le C\, |\T^\alpha|.$$
Indeed, assume $r\le\cos\alpha$, so that $\bar A = A$. Since $\int_{[A,B]} g^\alpha =0$, we have
\[\begin{split}
&\int_{[A,B]} g^\alpha v_{r,\alpha}=\int_{[A,B]} g^\alpha \Big(\frac{\sin (2\alpha)}{8\alpha}\ln(\cos^2 \alpha+y^2)-\frac {y^2}{4}\Big)dy\\
&=\int_{[A,B]} \Big (\frac{\sin(2\alpha)}{4\alpha} \frac{\cos(\alpha)}{\cos(\alpha)^2 + y^2} -\frac{\cos(\alpha)}{2} \Big)\Big(\frac{\sin (2\alpha)}{8\alpha}\ln(\cos^2 \alpha+y^2)-\frac {y^2}{4}\Big)dy,
\end{split}\]
and the bound comes by direct estimate.
 
Assume now, $r>\cos(\alpha)$. On $[\bar A, B]$, $v_{r,\alpha}$ has expression

$$ \forall y\in [r\sin(\alpha_0), \sin(\alpha)],\ v_{r,\alpha}(\cos(\alpha), y) = \frac{\sin(2\alpha)}{8\alpha}\ln\left(\frac{\cos(\alpha)^2 + y^2}{r^2}\right) + \frac{r^2 - (\cos(\alpha)^2+y^2)}4,$$

and since the function $r\mapsto r|\ln(r)|$ remains bounded for $r<r^*$, we can bound,
$$ |v_{r,\alpha}(\cos(\alpha), y)|\le \frac{\sin(\alpha)}{2\alpha} r|\ln(r)| + \frac{1-r^2}4 \le C.$$

As a consequence, we obtain 
$$\Big|\int_{[\bar A, B]} g^\alpha\, v_{r,\alpha}\Big| \le C \int_{[A,B]}|g^\alpha|, $$

and a direct estimate gives 
$$ \int_{[A,B]}|g^\alpha| \le \frac{\sin(2\alpha)}2 = 2|\T^\alpha|. $$

Consequently, by bootstrap, we obtain, for some constant positive $\overline C$, for any $0<r<r^*$ and $\alpha\in (0,\pi/2)$
$$ \Big|\int_{\T^\alpha} \psi_{r,\alpha}\Big| \le \overline C |\T^\alpha|^{3/2}.$$

This proves the first and second assertions.

Finally, noticing that 
$$ \int_{\T^\alpha} |\nabla v_{r,\alpha}|^2=H_r(\alpha) + \int_{[\bar A,B]} g^\alpha \,v_{r,\alpha},$$
we deduce from Lemma~\ref{lem:limitH} that
$$ 
\int_{\T^\alpha} \frac{|\nabla v_{r,\alpha}|^2}{|\ln(r)|} \le C\, |\T^\alpha|.
$$
We combine this with the bound on $\|\psi_{r,\alpha}\|_{H^1(\T^\alpha)}$ to obtain 
$$
\int_{\T^\alpha} \frac{|\nabla w_{r,\alpha}|^2}{|\ln(r)|} \le C\, |\T^\alpha|,
$$
which proves the third and fourth assertion.
\end{proof}

\noindent {\it Proof of Proposition~\ref{prop:Convergence}}, assertion (1).
Relying on the first assertions of Lemmas~\ref{lem:limitH} and \ref{lem:zeroOrderBounds}, respectively, there exists a constant $\overline C$ such that for $0<r<r^*$ and $\alpha \in (0,\pi/2)$, it holds
$$ \big|J_r(\alpha)-|\ln(r)|\,J_0(\alpha)\big| \le \|w_{r,\alpha} - v_{r,\alpha}\|_{L^1(\T^\alpha)} + \big| H_r(\alpha) -|\ln(r)|\, J_0(\alpha)\big | \le \overline C |\T^\alpha|.$$

\subsection{The second order convergence}

Let us fix two constants $0<\alpha_m<\alpha_M<\pi/2$ and $\alpha \in (\alpha_m,\alpha_M)$. We also suppose that $r<r^* < \cos(\alpha_M)$. In this setting, the disk $B(O,r)$ does not overlap with the edge $[A,B]$ of the triangle $\T^\alpha$.

To prove the second assertion of Proposition~\ref{prop:Convergence}, we are going this time to compare the derivatives of $J_r$ to those of $H_r$.

We first introduce
$$ \bar J_r(\T^\alpha\sm B(O,r)) = \frac12\int_{\T^\alpha\sm B(O,r)} |\nabla w_{r,\alpha}|^2,$$
$$ \bar H_r(\T^\alpha\sm B(O,r)) = \frac12\int_{\T^\alpha\sm B(O,r)} |\nabla v_{r,\alpha}|^2.$$
Note that $ J_r(\alpha) = 2\bar J_r(\T^\alpha\sm B(O,r)) $ and $H_r(\alpha) = 2\bar H_r(\T^\alpha\sm B(O,r)) - \int_{[A,B]} g^\alpha v_{r,\alpha} $.

To compute the derivatives of the functions $\alpha \mapsto \bar J_r(\T^\alpha\sm B(O,r))$ and $\alpha \mapsto \bar H_r(\T^\alpha\sm B(O,r)) $ we are going to express them with the formalism of distributed Eulerian shape derivatives in the spirit of Laurain in \cite[Sections 6-7]{L20}. 

We introduce for $\alpha \in (\alpha_m,\alpha_M)$ and $t \in (-\tau, \tau)$ such that $(1-\tau)\alpha, (1+\tau)\alpha \in (\alpha_m,\alpha_M)$,the transformation $T_{r,\alpha}^t :\Rtwo\to\Rtwo$ such that for $t\in(-\tau,\tau)$, $\T^{(1+t)\alpha} = T_{r,\alpha}^t(\T^\alpha\sm B(O,r))$. In polar coordinates, for $(\phi,\rho) \in \T^\alpha\sm B(O,r)$ we set
$$ T_{r,\alpha}^t(\phi,\ \rho ) = \left((1+t)\phi,\ \frac{\frac{\cos((1+t)\alpha)}{\cos((1+t)\phi)} - r}{\frac{\cos(\alpha)}{\cos(\phi)} - r} \, (\rho - r)\,+r\right). $$

We denote
$$V_{r,\alpha} = \frac{\d}{\d t}\Big|_{t = 0} T^t_{r,\alpha},$$
$$W_{r,\alpha} = \frac{\d}{\d t}\Big|_{t = 0} V_{r,(1+t)\alpha} \circ T^t_{r,\alpha}. $$
For $(\phi,\rho) \in \T^\alpha\sm B(O,r)$ we have
$$ V_{r,\alpha}(\phi,\rho) = \phi\rho\, \big[-\sin(\phi),\,\cos(\phi)\big] + (\rho - r) \frac{\phi \tan(\phi)\cos(\alpha)-\alpha\sin(\alpha)}{\cos(\alpha)-r\cos(\phi)} \big[\cos(\phi),\, \sin(\phi)\big] ; $$

$$ \begin{aligned}
W_{r,\alpha}(\phi,\rho) =& \Big(\phi\rho+2\phi (\rho-r)\frac{\phi \tan(\phi)\cos(\alpha)-\alpha\sin(\alpha)}{\cos(\alpha)-r\cos(\phi)} \Big)\,\big[-\sin(\phi), \cos(\phi)\big]\\
& + \left(-\phi^2\rho + (\rho-r) \frac{(\phi^2-\alpha^2)\cos(\alpha) + (1+2\phi\tan(\phi))(\phi\tan(\phi)\cos(\alpha)-\alpha\sin(\alpha))}{\cos(\alpha) - r\cos(\phi)}\right)\\
& [\cos(\phi), \sin(\phi)] .
\end{aligned}$$

To compute the second derivatives of $J_r$ and $H_r$ we differentiate the functions
$$(s,t)\mapsto \bar J_r(T^s_{r,(1+t)\alpha}\circ T^r_{r,\alpha}(\T^\alpha\sm B(O,r))),$$
$$(s,t)\mapsto \bar H_r(T^s_{r,(1+t)\alpha}\circ T^t_{r,\alpha}(\T^\alpha\sm B(O,r))). $$

Performing the first derivative, yields 
$$ \frac{\d}{\d s}\Big|_{s=0} \bar J_r(T^s_{r,(1+t)\alpha}\circ T^t_{r,\alpha}(\T^\alpha\sm B(O,r))) = D_E\bar J_r(T^t_{r,\alpha}(\T^\alpha\sm B(O,r)))(V_{r, (1+t)\alpha}) = \frac12(1+t)\alpha J_r'((1+t)\alpha), $$

$$ \frac{\d}{\d s}\Big|_{s=0} \bar H_r(T^s_{r,(1+t)\alpha}\circ T^t_{r,\alpha}(\T^\alpha\sm B(O,r))) = D_E\bar H_r(T^t_{r,\alpha}(\T^\alpha\sm B(O,r)))(V_{r, (1+t)\alpha}) = \frac12(1+t)\alpha \tilde H_r'((1+t)\alpha), $$
where $\tilde H_r(\alpha) = H_r(\alpha) + \int_{[A,B]} g^\alpha v_{r,\alpha}$ and $D_E\bar J_r$ and $D_E\bar H_r$ denote the Eulerian shape derivatives of the functionals $\bar J_r$ and $\bar H_r$.

Differentiating a second time, yields
$$ \frac{\d}{\d t}\Big|_{t=0} D_E\bar J_r(T^t_{r,\alpha}(\T^\alpha\sm B(O,r)))(V_{r,(1+t)\alpha}) = \frac12\alpha J_r'(\alpha) + \frac12\alpha^2 J_r''(\alpha), $$
$$ \frac{\d}{\d t}\Big|_{t=0} D_E\bar H_r(T^t_{r,\alpha}(\T^\alpha\sm B(O,r)))(V_{r,(1+t)\alpha}) = \frac12\alpha \tilde H_r'(\alpha) + \frac12\alpha^2 \tilde H_r''(\alpha). $$

To compute the first Eulerian shape derivatives, note that up to renaming $\alpha = (1+t)\alpha$ it is enough to compute 
$$ \frac{\d}{\d s}\Big|_{s=0} \bar J_r(T^s_{r,\alpha}(\T^\alpha\sm B(O,r))), $$
$$ \frac{\d}{\d s}\Big|_{s=0} \bar H_r(T^s_{r,\alpha}(\T^\alpha\sm B(O,r))). $$

We follow the computations of Laurain \cite[Proposition 6]{L20}, taking into account that $v_{r,\alpha}$ has a boundary term.
Let $u\in\{w,v\}$. Applying the change of variable, $x\mapsto T^s_{r,\alpha}(x)$ we obtain 
\be\label{eq:changeofvariable}
\begin{aligned}
\frac 12 \int_{T^s_{r,\alpha}(\T^\alpha\sm B(O,r))} |\nabla u_{r,(1+s)\alpha}|^2 = \frac12\int_{\T^\alpha\sm B(O,r)} \det (DT^s_{r,\alpha}) |\nabla u_{r,(1+s)\alpha}|^2\circ T^s_{r,\alpha}\\
=\frac12\int_{\T^\alpha\sm B(O,r)} \det(DT^s_{r,\alpha}) (DT^s_{r,\alpha})^{-1}((DT^s_{r,\alpha})^{-1})^T \nabla(u_{r,(1+s)\alpha}\circ T^s_{r,\alpha})\cdot \nabla(u_{r,(1+s)\alpha}\circ T^s_{r,\alpha}).
\end{aligned}
\ee
Here by $DG$ we denote the Jacobian matrix of a vector field $G:\Rtwo\to \Rtwo$ and by $M^T$ the transpose of a matrix $M \in \R^{2\times 2}$.

Note that for $s$ sufficiently small, $\det (DT^s_{r,\alpha}) = |\det (DT^s_{r,\alpha})|$.

Let us focus now on the two functions $w^t_{r,\alpha} = w_{r,(1+t)\alpha}\circ T^t_{r,\alpha}$ and $v^t_{r,\alpha} = v_{r,(1+t)}\circ T^t_{r,\alpha}$. Following \cite[Section 2.28-2.29]{SZ92} adapted to our specific setting, we show that they are differentiable at $0$ in the sense that the two functions $t^{-1}(w^t_{r,\alpha}-w_{r,\alpha})$ and $t^{-1}(v^t_{r,\alpha}-v_{r,\alpha})$ converge respectively strongly in $H^1$ to the functions $\dot w_{r,\alpha}$ and $\dot v_{r,\alpha}$, the material shape derivatives of $w_{r,\alpha}$ and $v_{r,\alpha}$. The two functions $\dot w_{r,\alpha}$ and $\dot v_{r,\alpha}$ satisfy, respectively, for any test function $u \in H^1(\T^\alpha \sm B(O,r))$ such that $u = 0 $ on $\partial B(O,r)$, 

$$ \int_{\T^\alpha \sm B(O,r)}\nabla\dot w\cdot \nabla u = -\int_{\T^\alpha\sm B(O,r)} A'_{r,\alpha} \nabla w\cdot \nabla u + \int_{\T^\alpha\sm B(O,r)}\div(V_{r,\alpha}) u, $$
$$ \int_{\T^\alpha \sm B(O,r)}\nabla\dot v\cdot \nabla u = -\int_{\T^\alpha\sm B(O,r)} A'_{r,\alpha} \nabla v\cdot \nabla u + \int_{\T^\alpha\sm B(O,r)}\div(V_{r,\alpha}) u + \int_{[A,B]} g_\alpha' u.$$
Above,
$$ A_{r,\alpha}' = -DV_{r,\alpha} -DV_{r,\alpha}^T + \div(V_{r,\alpha})I_2,$$
$$ g'_\alpha(y) = \alpha\sin(\alpha) -\frac{\cos(\alpha)}2 +\frac{\cos(2\alpha)}2 \frac{\cos(\alpha)}{\cos(\alpha)^2 + y^2} -y\arctan\Big(\frac y{\cos(\alpha)}\Big).$$

Concerning the function $w^t_{r,\alpha}$. Applying the change of variable on $x \to T^t_{r,\alpha}(x)$ in the equation satisfied by $w_{r, (1+t)\alpha}$ we find that $w^t_{r,\alpha}$ satisfies for any $u \in H^1(\T^\alpha\sm B(O,r))$ such that $u = 0$ on $\partial B(O,r)$, 
$$ \int_{\T^\alpha} A(t)\nabla w^t_{r,\alpha}\cdot \nabla u = \int_{\T^\alpha} \det(DT^t_{r,\alpha}) u$$
with 
$$ A(t) = \det(DT^t_{r,\alpha}) (DT^t_{r,\alpha})^{-1} ((DT^t_{r,\alpha})^{-1})^T. $$ 
Note that $t\mapsto A(t)$ is smooth and satisfies $A(0) = I_2$ and 
$$\frac{\d A(t)}{\d t}\big |_{t=0} = A'_{r,\alpha}.$$
Similarly $t \mapsto \det(DT^t_{r,\alpha})$ is smooth and satisfies $ \det(DT^0_{r,\alpha}) =1$ and 
$$ \frac{\d}{\d t}\Big|_{t=0}\det(DT^t_{r,\alpha}) = \div(V_{r,\alpha}).$$
Therefore inserting $u = w^t_{r,\alpha}$, we have for $t$ small enough, 
$$ \int_{\T^\alpha\sm B(O,r)} |\nabla w^t_{r,\alpha}|^2 \le C$$
so that $w^t_{r,\alpha}$ converges weakly in $H^1(\T^\alpha\sm B(O,r))$ to some function $w^*$ (up to extracting a subsequence), and passing to the limit in the equation of $w^t_{r,\alpha}$ we have that $w^* = w_{r,\alpha}$. Also, we have 
$$ \int_{\T^\alpha\sm B(O,r)} |\nabla w^t_{r,\alpha}|^2 = \int_{\T^\alpha\sm B(O,r)} (I_2 - A(t) )\nabla w^t_{r,\alpha} \cdot \nabla w_{r,\alpha}^t + \int_{\T^\alpha\sm B(O,r)} \det(DT^t_{r,\alpha}) w^t_{r,\alpha},$$
so that 

$$ \Big|\int_{\T^\alpha\sm B(O,r)} |\nabla w^t_{r,\alpha}|^2 - \int_{\T^\alpha\sm B(O,r)} \det(DT^t_{r,\alpha}) w^t_{r,\alpha}\Big|\le C \|A(t) - I\|_{L^\infty}$$

and since $\|A(t) - I\|_{L^\infty} \to 0$, we deduce, 
$$ \lim_{t\to 0} \int_{\T^\alpha\sm B(O,r)} |\nabla w^t_{r,\alpha}|^2 $$
$$= \lim_{t\to0} \int_{\T^\alpha\sm B(O,r)} \det(DT^t_{r,\alpha}) w^t_{r,\alpha} = \int_{\T^\alpha\sm B(O,r)} w_{r,\alpha} = \int_{\T^\alpha\sm B(O,r)} |\nabla w_{r,\alpha}|^2.$$
From these computations we conclude that $w^t_{r,\alpha}$ converges strongly in $H^1(\T^\alpha\sm B(O,r))$ to $w_{r,\alpha}$.

Now consider the function $t^{-1}(w^t_{r,\alpha} - w_{r,\alpha})$. It satisfies for any test function $u$ 
$$ \int_{\T^\alpha\sm B(O,r)} \nabla t^{-1}(w^t_{r,\alpha} - w_{r,\alpha}) \cdot \nabla u $$
$$= \int_{\T^\alpha\sm B(O,r)} t^{-1}(I_2-A(t)) \nabla w^t_{r,\alpha}\cdot \nabla u + \int_{\T^\alpha\sm B(O,r)} t^{-1}(\det(DT^t_{r,\alpha}) - 1) u.$$
Inserting $u = t^{-1}(w^t_{r,\alpha} - w_{r,\alpha})$, since $t^{-1}(I_2-A(t)) \to -A'$ and $t^{-1}(\det(DT^t_{r,\alpha}) - 1) \to \div V_{r,\alpha}$ in $L^\infty$ and $(w^t_{r,\alpha})_t$ is bounded in $H^1$, we can bound 
$$
\int_{\T^\alpha\sm B(O,r)} |\nabla t^{-1}(w^t_{r,\alpha} - w_{r,\alpha})|^2 \le C,
$$
so that $ t^{-1}(w^t_{r,\alpha} - w_{r,\alpha})$ converges weakly in $H^1$ (up to extracting a subsequence) to some function $\dot w^*$ and passing to the limit against any test function $u$, we deduce $\dot w^* = \dot w_{r,\alpha}$. 

Now, since the convergence of $w^t_{r,\alpha}$ to $w_{r,\alpha}$ is strong $H^1$, we obtain 
$$ \lim_{t\to 0} \int_{\T^\alpha\sm B(O,r)} |\nabla t^{-1}(w^t_{r,\alpha} - w_{r,\alpha})|^2$$
$$ = -\int_{\T^\alpha\sm B(O,r)} A'_{r,\alpha}\nabla w_{r,\alpha}\cdot \nabla \dot w_{r,\alpha} + \int_{\T^\alpha\sm B(O,r)} \div(V_{r,\alpha}) \dot w_{r,\alpha}= \int_{\T^\alpha\sm B(O,r)} |\nabla \dot w_{r,\alpha}|^2, $$
so that the convergence is actually strong in $H^1$.

The function $(t,x)\mapsto v^t_{r,\alpha}(x)$ has an explicit expression and belongs to $\mathcal C^\infty([-\tau,\tau]\times\T^\alpha\sm B(O,r))$, hence it is differentiable at $t=0$. The equation satisfied by $\dot v_{r,\alpha}$ comes from differentiating the equation satisfied by $v^t_{r,\alpha}$ with respect to $t$ in a similar way as for $w_{r,\alpha}$.

Back to the computation of our derivative, we can then differentiate the right hand side in Equation~\eqref{eq:changeofvariable}, and obtain, 
$$ \frac{\d}{\d s}\Big|_{s = 0} \frac12\int_{T^s_{r,\alpha}(\T^\alpha)} |\nabla u_{r,\alpha}|^2 = \frac12\int_\T^\alpha A' \nabla u_{r,\alpha} \cdot \nabla u_{r,\alpha} + \int \nabla u_{r,\alpha}\cdot \nabla \dot u_{r,\alpha}. $$

Now testing $\dot w_{r,\alpha}$ and $\dot v_{r,\alpha}$ against $w_{r,\alpha}$ and $v_{r,\alpha}$ in their respective equations, we deduce,
$$D_E\bar J_r(\T^\alpha\sm B(O,r))(V_{r,\alpha}) = \int_{\T^{\alpha}\sm B(O,r)}S(w_{r,\alpha}):DV_{r,\alpha},$$ 
$$\begin{aligned}D_E\bar H_r(\T^\alpha\sm B(O,r))(V_{r,\alpha}) =&\int_{\T^{\alpha}\sm B(O,r)}S(v_{r,\alpha}):DV_{r,\alpha} \\
&+\int_{[A,B]} g_{\alpha}' v_{r,\alpha}, 
\end{aligned}$$
with 
$$ S(u) = \left(-\frac12 |\nabla u|^2 + u \right) I_2 + \nabla u\otimes \nabla u.$$

To compute the derivatives of $D_E\bar J_r(T^t_{r,\alpha}(\T^\alpha\sm B(O,r)))(V_{r,(1+t)\alpha})$ and $D_E\bar H_r(T^t_{r,\alpha}(\T^\alpha\sm B(O,r)))(V_{r,(1+t)\alpha})$ we also follow the main ideas of Laurain in \cite[Proposition 7]{L20}, but here we have to take into account that $V_{r,(1+t)\alpha}$ depends on the time $t$. 

We first take care of the volume term. Applying the change of variable, $x\mapsto T^t_{r,\alpha}(x)$, we obtain, for $u \in \{v,w\}$,
$$ \int_{\T^{(1+t)\alpha}\sm B(O,r)} S(u_{r,(1+t)\alpha}) : DV_{r,(1+t)\alpha}$$
$$=\int_{\T^\alpha\sm B(O,r)}\det(DT^t_{r,\alpha}) S(u_{r,(1+t)\alpha})\circ T^t_{r,\alpha} : DV_{r,(1+t)\alpha}\circ T^t_{r,\alpha} $$
$$=\int_{\T^\alpha\sm B(O,r)}\det(DT^t_{r,\alpha}) S(u_{r,(1+t)\alpha})\circ T^t_{r,\alpha} : D(V_{r,(1+t)\alpha}\circ T^t_{r,\alpha})(DT^t_{r,\alpha})^{-1}.$$

Differentiating the right hand side, we find 
$$ 
\frac{\d}{\d t}\Big|_{t=0} \int_{\T^{(1+t)\alpha}\sm B(O,r)} S(u_{r,(1+t)\alpha}): DV_{r,(1+t)\alpha} $$
$$=\int_{\T^\alpha} \div V_{r,\alpha} S(u_{r,\alpha}):DV_{r,\alpha} + \int_{\T^\alpha} S'_u:DV_{r,\alpha} + \int_{\T^\alpha}S(u_{r,\alpha}):DW_{r,\alpha} - \int_{\T^\alpha} S(u_{r,\alpha}) : (DV_{r,\alpha})^2,
$$
where $S'_u$ has expression, 
$$\begin{aligned}
S'_u =& (-\nabla u_{r,\alpha}\cdot \nabla \dot u_{r,\alpha}+ DV_{r,\alpha}\nabla u_{r,\alpha}\cdot \nabla u_{r,\alpha}+ \dot u_{r,\alpha}) I_2\\
& + \nabla u_{r,\alpha}\otimes \nabla \dot u_{r,\alpha} + \nabla \dot u_{r,\alpha} \otimes \nabla u_{r,\alpha} - (DV_{r,\alpha}^T\nabla u_{r,\alpha} \otimes \nabla u_{r,\alpha} + \nabla u \otimes DV_{r,\alpha}^T\nabla u).
\end{aligned}$$

Now considering the border term, since $g_\alpha'$ is explicit and that $\int_{[A,B]}g_\alpha' = 0$ we may express it
$$\begin{aligned}
\int_{[A,B]}g_{\alpha}' v_{r,\alpha} = \int_0^{\sin(\alpha)} &\Big(\alpha\sin(\alpha) - \frac{\cos(\alpha)}2 - y \arctan\Big(\frac y{\cos(\alpha)}\Big) + \frac{\cos(2\alpha)}{2\alpha}\frac{\cos(\alpha)}{\cos(\alpha)^2 + y^2}\Big) \\
 &\Big(\frac{\sin(2\alpha)}{8\alpha}\ln(\cos(\alpha)^2 + y^2) -\frac{y^2}4\Big) \d y.
\end{aligned}$$
The function, $\alpha \to \int_{[A,B]} g_\alpha' v_{r,\alpha}$ is in $\mathcal C^\infty((0,\pi/2))$ so that it remains bounded as well as its derivative for $\alpha \in [\alpha_m,\alpha_M]$. Consequently, since it holds 
$$ \frac{\d}{\d t}\Big|_{t = 0 }\int_{[A_t,B_t]} g_{(1+t)\alpha}' v_{r,(1+t)\alpha} = \alpha \frac{\d}{\d \alpha}\int_{[A,B]}g_{\alpha}' v_{r,\alpha}, $$
we deduce that this term also remains bounded for $\alpha \in (\alpha_m,\alpha_M)$.

Putting all of these together, we find that the difference $J_r'' - \tilde H_r''$ satisfies for any $\alpha \in [\alpha_m,\alpha_M]$,
$$ 
\begin{aligned}
\alpha^2 J_r''(\alpha) - \alpha^2 \tilde H_r''(\alpha) =& \int_{[A,B]} g_\alpha'v_{r,\alpha} - \frac{\d}{\d t}\Big|_{t = 0 }\int_{[A_t,B_t]}g_{(1+t)\alpha}' v_{r,(1+t)\alpha}\\
& +\int_{\T^\alpha} M_{r,\alpha}^1\nabla w_{r,\alpha}\cdot \nabla w_{r,\alpha} - \int_{\T^\alpha} M_{r,\alpha}^1\nabla v_{r,\alpha}\cdot \nabla v_{r,\alpha}\\
& + \int_{\T^\alpha} M_{r,\alpha}^2 \nabla w_{r,\alpha}\cdot\nabla \dot w_{r,\alpha} - \int_{\T^\alpha} M_{r,\alpha}^2 \nabla v_{r,\alpha}\cdot\nabla \dot v_{r,\alpha}\\
& + \int_{\T^\alpha} \Phi_{r,\alpha}^1 (w_{r,\alpha} - v_{r,\alpha}) + \int_{\T^\alpha} \Phi_{r,\alpha}^2 (\dot w_{r,\alpha} - \dot v_{r,\alpha}).
\end{aligned}
$$
Above, the expressions of the coefficients of the two matrices $M_{r,\alpha}^1$ and $M_{r,\alpha}^2$ and the two functions $\Phi_{r,\alpha}^1$ and $\Phi_{r,\alpha}^2$ depend only on products and sums of the first spatial derivatives of $V_{r,\alpha}$ and $W_{r,\alpha}$.

Computing directly the terms $\partial_i V_{r,\alpha}^j$ and $\partial_i W_{r,\alpha}^j$ for $i,j\in\{x,y\}$, we find that they remain bounded for $\alpha\in[\alpha_m,\alpha_M]$ and $r< r^*$ independently of $\alpha$ and $r$. Also, since the two integral terms on the segment $[A,B]$ are bounded independently of $\alpha$ and $r$, for a constant $C(\alpha_m, \alpha_M)$ depending only on $\alpha_m$ and $\alpha_M$ we get

$$ 
\begin{aligned}
\alpha^2 |J_r''(\alpha) - \tilde H_r''(\alpha)| \le C(\alpha_m,\alpha_M) \big [ 1 &+ \|\nabla w_{r,\alpha} - \nabla v_{r,\alpha}\|_{L^2} (\|\nabla w_{r,\alpha}\|_{L^2} + \|\nabla v_{r,\alpha}\|_{L^2} + \|\nabla \dot v_{r,\alpha}\|_{L^2})\\
&+ \|\nabla \dot w_{r,\alpha} -\nabla \dot v_{r,\alpha}\|_{L^2}\|\nabla w_{r,\alpha}\|_{L^2}\\
&+ \|w_{r,\alpha} - v_{r,\alpha}\|_{L^1} + \|\dot w_{r,\alpha} - \dot v_{r,\alpha}\|_{L^1} \big ].
\end{aligned}
$$

Now since $g^\alpha$ is explicit and $\int_{[A,B]} g^\alpha = 0$, we can express,
$$ \int_{[A,B]} g^\alpha v_{r,\alpha} = \int_0^{\sin(\alpha)}\Big(\frac{\sin(2\alpha)}{4\alpha}\frac{\cos(\alpha)}{\cos(\alpha)^2 + y^4} -\frac{\cos(\alpha)}2\Big)\Big(\frac{\sin(2\alpha)}{8\alpha}\ln(\cos(\alpha)^2+y^2)-\frac{y^2}{\cos(\alpha)^2}\Big)\d y$$
which is $C^\infty$ on $(0,\pi/2)$ so that its second derivative remains bounded on $(\alpha_m, \alpha_M)$.

Then we can bound
\be \label{eq:2ndboundsJH}
\begin{aligned}
|J_r''(\alpha) - H_r''(\alpha)| \le C(\alpha_m,\alpha_M) \big [ 1 &+ \|\nabla w_{r,\alpha} - \nabla v_{r,\alpha}\|_{L^2} (\|\nabla w_{r,\alpha}\|_{L^2} + \|\nabla v_{r,\alpha}\|_{L^2} + \|\nabla \dot v_{r,\alpha}\|_{L^2})\\
&+ \|\nabla \dot w_{r,\alpha} -\nabla \dot v_{r,\alpha}\|_{L^2}\|\nabla w_{r,\alpha}\|_{L^2}\\
&+ \|w_{r,\alpha} - v_{r,\alpha}\|_{L^1} + \|\dot w_{r,\alpha} - \dot v_{r,\alpha}\|_{L^1} \big ].
\end{aligned}
\ee

We now prove the following Lemma,
\begin{lem}\label{lem:SecondOrderBounds}
The following limits hold uniformly on $[\alpha_m,\alpha_M]$,
\begin{itemize}
    \item[(1)] $$ \lim_{r\to0} \frac{\|\dot w_{r,\alpha}-\dot v_{r,\alpha}\|_{H^1(\T^\alpha)}}{|\ln(r)|} = 0;$$
    \item[(2)]$$ \lim_{r\to0} \frac{\|\nabla (\dot w_{r,\alpha} - \dot v_{r,\alpha})\|_{L^2(\T^\alpha)}}{\sqrt{|\ln(r)|}} = 0.$$
\end{itemize}
And there exists a constant $C(\alpha_m,\alpha_M)$ such that for any $r<r^*$ and $\alpha\in(\alpha_m,\alpha_M)$,
\begin{itemize}
    \item[(3)] $$ \frac{\|\nabla \dot w_{r,\alpha}\|_{L^2(\T^\alpha)}}{\sqrt{|\ln(r)|}} \le C(\alpha_m,\alpha_M);$$
    \item[(4)] $$ \frac{\|\nabla \dot v_{r,\alpha}\|_{L^2(\T^\alpha)}}{\sqrt{|\ln(r)|}} \le C(\alpha_m,\alpha_M).$$
\end{itemize}
\end{lem}

\begin{proof}
 In the following, we denote by $C$ a positive constant depending only on $\alpha_m$, $\alpha_M$ and $r^*$ which may increase from line to line. Before starting the proof of this Lemma we get the following explicit control of the Poincar\'{e} constant:\\
{\it
There exists a constant $D$ such that for any $\alpha\in(\alpha_m,\alpha_M)$ and $r<r^*$ small enough, it holds for all $u \in H^1(\T^\alpha)$ satisfying $u = 0$ on $ B(O,r)$,
\be\label{eq:claim}
\int_{\T^\alpha} u^2 \le D |\ln(r)| \int_{\T^\alpha} |\nabla u|^2.
\ee
}

To prove the claim, we first notice that since $\alpha\in[\alpha_m,\alpha_M]$ the triangles $\T^\alpha$ are uniformly Lipschitz and therefore the extension operator $H^1(\T^\alpha) \hookrightarrow H^1(S(\alpha))$ from the triangle $\T^\alpha$ to the sector $S(\alpha)$ of angle $\alpha$ and radius $1$ is uniform in $\alpha$.

Hence, it is enough to estimate 
$$ \inf\Big \{ \frac{\int_{S(\alpha)} |\nabla u|^2}{\int_{S(\alpha)} u^2} ~\Big|~ u\in H^1(S(\alpha)),\ u=0\text{ on }B(O,r)\cap S(\alpha) \Big\}.$$

For any $r<r^*$, this infimum is attained for radially symmetric function $u_r$ that is independent of $\alpha$. 
Assume $u_r = C_r$ on $\partial B_1$ with $C_r$ such that 
$$ \int_{B_1\setminus B_r} u_r^2 = 1. $$
 Then for $h_r$ the harmonic function such that $h_r = 0 $ on $\partial B_r$ and $h_r = C_r$ on $\partial B_1$. It holds 
$$ \int_{B_1\setminus B_r} |\nabla h_r|^2 \le \int_{B_1\setminus B_r} |\nabla u_r|^2. $$
And since 
$$ h_r = C_r \left(1-\frac{\ln(|x|)}{\ln(r)}\right), $$
is explicit, one may compute 
$$ \int_{B_1\setminus B_r} |\nabla h_r|^2 = \frac{2\pi\,{C_r}^2}{|\ln(r)|}, $$
and,
$$ \int_{B_1\setminus B_r} h_r^2 = \pi\,{C_r}^2 \Big(1+\frac1{\ln(r)} + \frac{1-r^2}{2\ln(r)^2}\Big). $$

Then we obtain that for $r<r^*$
$$ \int_{B_1\setminus B_r}|\nabla u_r|^2\le \frac{\int_{B_1\setminus B_r}|\nabla h_r|^2}{\int_{B_1\setminus B_r} h_r^2} \le \frac{C}{|\ln(r)|}.$$
Hence, $u_r\to \bar C$ in $H^1(B_1)$ which implies $C_r \to \bar C$ with $\bar C^2|B_1| = 1$. In the same time
$$\frac{\int_{B_1\setminus B_r}|\nabla u_r|^2}{\int_{B_1\setminus B_r} u_r^2} \ge \frac{\int_{B_1\setminus B_r} |\nabla h_r|^2 }{1} =\frac{2\pi\,{C_r}^2}{|\ln(r)|}. $$
Since $C_r \to C$, this proves the claim.

Now, let us recall the equations satisfied by $\dot w_{r,\alpha}$ and $\dot v_{r,\alpha}$, for any $u\in H^1(\T^\alpha)$, $u = 0$ on $B(O,r)$,
$$
\int_{\T^\alpha} \nabla \dot w_{r,\alpha} \cdot \nabla u = - \int_{\T^\alpha } A_{r,\alpha}'\nabla w_{r,\alpha} \cdot \nabla u + \int_{\T^\alpha} \div(V_{r,\alpha})\ u,$$ 
$$ \int_{\T^\alpha} \nabla \dot v_{r,\alpha} \cdot \nabla u = - \int_{\T^\alpha} A_{r,\alpha}'\nabla v_{r,\alpha} \cdot \nabla u + \int_{\T^\alpha} \div(V_{r,\alpha})\ u + \int_{[A,B]} g_\alpha'\,u.
$$
By definition, $g_\alpha'$ is uniformly bounded for $\alpha \in (\alpha_m,\alpha_M)$. Also since the expressions of $A'_{r,\alpha}$ and $\div(V_{r,\alpha})$ only involve first spatial derivatives of $V_{r,\alpha}$, they are in $L^\infty(\T^\alpha)$ and bounded independently of $r$ and $\alpha$ for $\alpha\in[\alpha_m,\alpha_M]$ and $r<r^*$.

We deduce that the function $\dot w_{r,\alpha} - \dot v_{r,\alpha}$ satisfies for all $u$ in $H^1(\T^\alpha)$ and vanishing on $\T^\alpha\cap B(O,r)$,
\be
\int_{\T^\alpha} \nabla (\dot w_{r,\alpha} -\dot v_{r,\alpha}) \cdot \nabla u = - \int_{\T^\alpha} A_{r,\alpha}'\nabla (w_{r,\alpha}-v_{r,\alpha}) \cdot \nabla u-\int_{[A,B]} g_\alpha'\,u.
\ee

Then, since the square of the trace constant $C_T(\alpha)$ is continuous on $(0,\pi/2)$, it is bounded on $(\alpha_m,\alpha_M)$, and we can estimate
$$
\|\nabla(\dot w_{r,\alpha} -\dot v_{r,\alpha})\|_{L^2} \le C\big( \|\nabla (w_{r,\alpha} - v_{r,\alpha})\|_{L^2}+\left(1 + D| \ln(r) |\right)^{1/2} \big).
$$
As a consequence of Lemma~\ref{lem:zeroOrderBounds}, we deduce
\be
\frac{\|\nabla(\dot w_{r,\alpha} -\dot v_{r,\alpha})\|_{L^2}}{\sqrt{|\ln(r)|}} \le C
\ee
and also, from the Poincar\'e inequality,
$$
\frac{\|\dot w_{r,\alpha} -\dot v_{r,\alpha}\|_{L^2}^2}{\ln(r)^2} \le D \frac{\|\nabla(\dot w_{r,\alpha} -\dot v_{r,\alpha})\|_{L^2}^2}{|\ln(r)|} \le C.
$$

Therefore, we obtain that 
$$ \frac{\|\dot w_{r,\alpha}-\dot v_{r,\alpha}\|_{H^1(\T^\alpha)}}{|\ln(r)|} \le C, $$
with 
$$
\frac{\|\nabla(\dot w_{r,\alpha} -\dot v_{r,\alpha})\|_{L^2}}{|\ln(r)|} \to 0,
$$
so that up to extraction, there exists a constant $\overline C$ such that
\be
\frac{\dot w_{r,\alpha} - \dot v_{r,\alpha}}{|\ln(r)|} \to \overline{C},
\ee
strongly in $H^1$ and the convergence is uniform for $\alpha\in[\alpha_m,\alpha_M]$.

In order to compute $\overline{C}$ we use the equations satisfied by $w_{r,\alpha}$ and $\dot w_{r,\alpha} - \dot v_{r,\alpha}$ to obtain 
\be\label{eq:intdotwdotvNEW}
\int_{\T^\alpha} \dot w_{r,\alpha} - \dot v_{r,\alpha}=- \int_{\T^\alpha} A_{r,\alpha}' \nabla(w_{r,\alpha}-v_{r,\alpha}) \cdot\nabla w_{r,\alpha} - \int_{[A,B]} g_\alpha' (w_{r,\alpha}-v_{r,\alpha}) -\int_{[A,B]} g_\alpha'\, v_{r,\alpha}.
\ee

From a direct computation and the fact that $\int_{[A,B]}g' = 0$, we get that for any $\alpha \in [\alpha_m,\alpha_M]$ and $r<r^*$,
$$
\Big| \int_{[A,B]} g_\alpha'\, v_{r,\alpha} \Big|\le C.
$$

Dividing by $|\ln(r)|$ in Equation \eqref{eq:intdotwdotvNEW} and taking the limit $r \to 0$ we find $\overline{C} = 0$ and this proves the first assertion. \\
From the equation satisfied by $\dot w_{r,\alpha} - \dot v_{r,\alpha}$ we also obtain 
$$
\frac{\|\nabla(\dot w_{r,\alpha} -\dot v_{r,\alpha})\|_{L^2}^2}{|\ln(r)|} = - \int_{\T^\alpha \setminus B(O,r)} A'_{r,\alpha} \nabla \frac{w_{r,\alpha}-v_{r,\alpha}}{\sqrt{|\ln(r)|}}\cdot\nabla \frac{\dot w_{r,\alpha} - \dot v_{r,\alpha}}{\sqrt{|\ln(r)|}} - \int_{[A,B]} g_\alpha' \frac{(\dot w_{r,\alpha} - \dot v_{r,\alpha})}{|\ln(r)|}
$$
which gives, passing to the limit,
\be\label{eq:limnabladotwv}
\frac{\|\nabla\dot w_{r,\alpha} -\nabla\dot v_{r,\alpha}\|_{L^2}}{\sqrt{|\ln(r)|}} \to 0,
\ee
the convergence being uniform for $\alpha\in[\alpha_m,\alpha_M]$, which proves the second assertion.

Now, from the equation satisfied by $\dot w_{r,\alpha}$ , compute 
$$ \int_{\T^\alpha} |\nabla w_{r,\alpha}|^2 = -\int_{\T^\alpha}A'_{r,\alpha} \nabla w_{r,\alpha}\cdot \nabla \dot w_{r,\alpha} + \int_{\T^\alpha} \div(V_{r,\alpha}) \dot w. $$
Then, we can estimate as a consequence of Lemma~\ref{lem:zeroOrderBounds} and the uniform Poincar\'e estimate~\ref{eq:claim} for some constant $C(\alpha_m,\alpha_M)$, for any $\alpha\in[\alpha_m,\alpha_M]$ and $r<r^*$ small,
$$ \int_{\T^\alpha} \frac{|\nabla \dot w_{r,\alpha}|}{|\ln(r)|} \le C(\alpha_m,\alpha_M). $$
And by the uniform limit given in Equation~\ref{eq:limnabladotwv}, we have that the same bounds holds for $\dot v_{r,\alpha}$, 
$$ \int_{\T^\alpha} \frac{|\nabla \dot v_{r,\alpha}|}{|\ln(r)|} \le C(\alpha_m,\alpha_M). $$

This proves the third and fourth assertion.
\end{proof}

We may now finish the proof of the second assertion of Proposition~\ref{prop:Convergence}
\begin{proof}
As a direct consequence of Lemma~\ref{lem:SecondOrderBounds}, Lemma~\ref{lem:zeroOrderBounds} and Equation~\eqref{eq:2ndboundsJH}, we deduce that uniformly on $(\alpha_m, \alpha_M)$, it holds 
$$\lim_{r\to 0} \frac{|J_r''(\alpha) - H_r''(\alpha)|}{|\ln(r)|} = 0.$$

And by the second assertion of Lemma~\ref{lem:limitH}, we also have that for $\alpha \in(\alpha_m,\alpha_M)$ and $r< r^*$,
$$ |H_r''(\alpha) - |\ln(r)| J_0''(\alpha)| \le C(\alpha_m, \alpha_M)$$

Putting the two together we obtain that uniformly on $(\alpha_m,\alpha_M)$, it holds

$$ \lim_{r\to 0} \Big|\frac{J_r''(\alpha)}{|\ln(r)|} - J_0''(\alpha) \Big|= 0 $$
\end{proof}

\section{Proof of Proposition~\ref{prop:maximalityF0}}\label{sec:proofPropMaxF0}

Note that $F_0(\pi/6,\pi/6)=\frac{3\sqrt3}{4\pi}$. Let us prove first the existence of a maximizer in $\S$ for $F_0$. Take a maximizing sequence $(\alpha^n_1,\alpha^n_2)_n\in\S$. We first observe that if $\alpha^n_1 + \alpha^n_2 \to 0$ then $\alpha^n_1\to0$ and $\alpha^n_2\to0$, so after suitable expansion of the ingredients of $F_0$, we get
\begin{equation}\label{asx01}
\lim_{n \to +\infty}F_0(\alpha^n_1, \alpha^n_2) \to \frac 14 < \frac{3\sqrt3}{4\pi}.
\end{equation}
As a consequence, there exists an $\alpha^*$ such that if $\alpha_1 + \alpha_2 < \alpha^*$, $F_0(\alpha_1, \alpha_2) < \frac{3\sqrt3}{4\pi} $.

Proving that
$$F_0(\alpha_1, \alpha_2)\le\frac{3\sqrt3}{4\pi}$$ 
is equivalent to
$$\tilde f(\alpha_1, \alpha_2)=f(\alpha_1) + f(\alpha_2) + f(\pi/2-(\alpha_1+\alpha_2)) \le 0$$
where $f : \alpha\mapsto \frac{\sin(2\alpha)^2}\alpha - \frac{3\sqrt3}\pi \sin(2\alpha)$. 
If $(\alpha^n_1,\alpha^n_2)_n\in\S$ is maximizing for $\tilde f$, three situations may occur.

\begin{itemize}
\item [(i)] If $\alpha^n_1, \alpha^n_2 \to 0$ then $\tilde f(\alpha^n_1, \alpha^n_2) \to 0$ so the assertion is proved. 

\item [(ii)] If $\alpha^n_1\to 0$, $\alpha_2^n \to a >0 $. Then
$$\tilde f(\alpha^n_1, \alpha^n_2) \to f(a)+ f(\pi/2-a).$$
For $a\in [0, \pi/4]$ elementary computations show that the maximum of $ f(a)+ f(\pi/2-a)$ equals $0$ and is attained at $a=0$, which contradicts our hypothesis $a>0$.

\item [(iii)] Assume $\alpha^n_1\to \overline \alpha_1>0$, $\alpha_2^n \to \overline \alpha_2 >0 $. If $ \overline\alpha_1= \overline \alpha_2= \pi/4$ then we are back to the previous case. If only $ \overline \alpha_1=\pi/4$ then $ \overline \alpha_2, \pi/2-( \overline \alpha_1+ \overline \alpha_2) \in (0, \pi/4)$, hence by symmetry we can assume $\overline \alpha_1, \overline \alpha_2 \in (0, \pi/4)$.
\end{itemize}

Computing the gradient of $(\alpha_1, \alpha_2) \to \tilde f(\alpha_1,\alpha_2) $ on $(0,\pi/4)^2$ and relaying on the maximality of $(\overline \alpha_1, \overline \alpha_2)$, denoting $\bar\alpha_3 = \pi/2 - (\bar\alpha_1 + \bar\alpha_2)$, we find
$$f'(\bar\alpha_1) = f'(\bar\alpha_2) = f'(\bar\alpha_3).$$
This implies that if there is an extrema $(\bar\alpha_1, \bar\alpha_2)$ with $\bar\alpha_1, \bar\alpha_2, \bar\alpha_3$ pairwise different, there must be two distinct points $\beta_1,\beta_2\in (0, \pi/2)$ such that 
$$f''(\beta_1) = f''(\beta_2)=0.$$
By elementary computation, the expression of $f''$ is 
$$f'' :\alpha \mapsto \frac{8\cos(4\alpha)}\alpha - \frac{4\sin(4\alpha)}{\alpha^2} + \frac{2\sin(2\alpha)^2}{\alpha^3} + \frac{12\sqrt3}\pi \sin(2\alpha). $$
The function $f''$ has a zero at $0$ and another one in $[0.76,0.77]$, is negative on $[0.01, 0.76]$ and positive on $[0.77, \pi/2]$. Computing $f'''$,
$$f''' : \alpha \mapsto -\frac{32 \sin(4\alpha)}\alpha - \frac{24\cos(4\alpha)}{\alpha^2} + \frac{12\sin(4\alpha)}{\alpha^3} - \frac{6\sin(2\alpha)^2}{\alpha^4} + \frac{24\sqrt3}\pi \cos(2\alpha), $$
it has its zeros in $[0.39,0.4]$ and $[1.28, 1.29]$ which implies that $f''$ has a unique $0$ in $(0,\pi/2)$. As a consequence, for any maxima $(\bar\alpha_1, \bar\alpha_2, \bar\alpha_3)$ two of the points are equal. It remains to study the maximum of a one dimensional problem and to show that for any $\alpha \in [0, \pi/4]$, 
$$h(\alpha)=2f(\alpha) + f(\pi/2 - 2\alpha) \le 0 $$
with equality only for $\alpha = 0$ or $\alpha = \pi/6$. This inequality is elementary.
The case $\alpha=0$ is out of our range.

As a consequence, $\tilde f$ is non positive on $[0, \pi/4]^2$ and vanishes only for $\bar \alpha_1=\bar\alpha_2=0$ and $\bar\alpha_1=\bar\alpha_2=\pi/6$. For the maximality of $F_0$, the case $\bar \alpha_1=\bar\alpha_2=0$ is eliminated by \eqref{asx01}. This concludes the proof. 
\qed

\section{Proof of Proposition~\ref{prop:MaxFr}}\label{sec:proofPropMaxFr}

\subsection{Proof of the existence}

We define $E_r$ the supremum of $F_r$ among all admissible couples $(\alpha_1,\alpha_2)\in\S$. We argue that for $r<r^*$ sufficiently small, a maximizer exists and is non degenerate. 

For a fixed $r>0$, $F_r$ is continuous on the set $\S$ and the set of non degenerate couples $\S_{DN,r}$ is a compact subset of $\S$. Therefore, for any $r>0$, $F_r$ admits a maximum on $\S_{DN,r}$.

Now, consider a sequence $(\alpha_{1,r}, \alpha_{2,r})$ of degenerate couples, that is, for all $r>0$, $0<\alpha_{1,r}+\alpha_{2,r} < \eta_r$, with $\sin(\eta_r) =r $. We denote by $\alpha_{3,r}$ the third angle, $\alpha_{3,r} = \pi/2-(\alpha_1+\alpha_2)$.

Notice that for any $\alpha \in [0,\pi/2)$, it holds, $|\T^{\alpha} \setminus B(O,r)| \ge (1-r)^2 |\T^\alpha|$. We can then compute, 

$$ \begin{array}{lll}
F_r(\alpha_{1,r}, \alpha_{2,r}) &\le\frac4{(1-r)^2}& \frac{J_r(\alpha_{1,r}) + J_r(\alpha_{2,r})+J_r(\pi/2 - (\alpha_{1,r}+\alpha_{2,r}))}{\sin(2\alpha_{1,r}) + \sin(2\alpha_{2,r}) + \sin(2(\alpha_{1,r}+\alpha_{2,r}))} \\
&= \frac4{(1-r)^2}& \left[ \frac{\sin(2\alpha_{1,r})+\sin(2\alpha_{2,r})}{\sin(2\alpha_{1,r}) + \sin(2\alpha_{2,r}) + \sin(2(\alpha_{1,r}+\alpha_{2,r}))} \frac{J_r(\alpha_{1,r}) + J_r(\alpha_{2,r})}{\sin(2\alpha_{1,r}) + \sin(2\alpha_{2,r})} \right.\\
&&+\left. \frac{\sin(2(\alpha_{1,r}+\alpha_{2,r}))}{\sin(2\alpha_{1,r}) + \sin(2\alpha_{2,r}) + \sin(2(\alpha_{1,r}+\alpha_{2,r}))}\frac{J_r(\pi/2-(\alpha_{1,r}+\alpha_{2,r}))}{\sin(2\alpha_{1,r}) + \sin(2\alpha_{2,r})} \right]\\
&\le \frac4{(1-r)^2}& \left[ \frac{\sin(2\alpha_{1,r})+\sin(2\alpha_{2,r})}{\sin(2\alpha_{1,r}) + \sin(2\alpha_{2,r}) + \sin(2(\alpha_{1,r}+\alpha_{2,r}))} \sup_{i = 1,2}\frac{J_r(\alpha_{i,r}) }{\sin(2\alpha_{i,r})}\right.\\
&& +\left. \frac{\sin(2(\alpha_{1,r}+\alpha_{2,r}))}{\sin(2\alpha_{1,r}) + \sin(2\alpha_{2,r}) + \sin(2(\alpha_{1,r}+\alpha_{2,r}))}\frac{J_r(\pi/2-(\alpha_{1,r}+\alpha_{2,r}))}{\sin(2(\alpha_{1,r}+\alpha_{2,r}))} \right]
\end{array}$$

If, $\eta>0$ is small enough, it holds for any $\alpha\le 2\eta$, $\alpha(1- \eta)\le\sin(\alpha)\le\alpha$, in consequence, for $r$ small enough, we obtain,

$$ F_r(\alpha_{1,r}, \alpha_{2,r}) \le \frac{2}{(1-r)^2(1-\eta_r)} \left[ \sup_{i = 1,2}\frac{J_r(\alpha_{i,r}) }{\sin(2\alpha_{i,r})} + \frac{J_r(\pi/2-(\alpha_{1,r}+\alpha_{2,r}))}{\sin(2(\alpha_{1,r}+\alpha_{2,r}))}\right]. $$

From the first assertion of Proposition~\ref{prop:Convergence}, we know that there is an $r^*$ sufficiently small, independent of $\alpha_{1,r},\alpha_{2,r}, \alpha_{3,r}$ such that for $i= 1,2,3$, 
$$\frac1{|\ln(r)|} \frac{J_r(\alpha_{i,r})}{|\T^{\alpha_{i,r}}|} \le \frac{J_0(\alpha_{i,r})}{|\T^{\alpha_{i,r}}|} + \eps. $$

Additionally, we know that
$$\lim_{\eta\to 0}\frac{J_0(\eta)}{|\T^{\eta}|} = \frac12,$$
and,
$$\lim_{\eta\to0} \frac{J_0(\pi/2-\eta)}{|\T^{\pi/2-\eta}|} = 0$$
and consequently for $r$ small, for $i=1,2$, 
$$\frac1{|\ln(r)|} \frac{J_r(\alpha_{i,r})}{|\T^{\alpha_{i,r}}|} \le \frac12 + \eps, $$
and,
$$\frac1{|\ln(r)|} \frac{J_r(\alpha_{3,r})}{|\T^{\alpha_{3,r}}|} \le\eps.$$

As a consequence, we deduce that for any $\eps >0$ for $r$ small enough it holds,
$$\frac{F_r(\alpha_{1,r}, \alpha_{2,r})}{|\ln(r)|} \le \frac14 (1+ \eps).$$
This bound is uniform and holds for any sequence of degenerate couples $(\alpha_{1,r}, \alpha_{2,r})$. So that, for any $\eps>0$ and $r<r^*$ small enough, we have
$$\sup_{\S\setminus \S_{ND,r}} \frac{F_r}{|\ln(r)|} \le \frac14 (1+\eps).$$
and in addition, 
$$\left|\frac{F_r(\pi/6, \pi/6)}{|\ln(r)|} - E_0\right| < \eps, $$
with $E_0 > 1/4$.

In conclusion, if $r<r^*$, the supremum $E_r$ of $F_r$ on $\S$ is attained and the maximizers are non degenerate.

\subsection{Proof of the limits}
Firstly, by definition, for all $(\alpha_1,\alpha_2)\in\S$,
$$ E_r \ge F_r(\alpha_1, \alpha_2). $$
Then, dividing by $|\ln(r)|$ and passing to the limit, we obtain 
$$ \liminf_{r\to0} \frac{E_r}{|\ln(r)|} \ge F_0(\alpha_1, \alpha_2), $$
and therefore,
$$\liminf_{r\to0} \frac{E_r}{|\ln(r)|} \ge E_0.$$

Now consider $(\alpha_{1,r},\alpha_{2,r})_r$ a sequence maximizing $F_r$ for all $r$, since it is bounded it converges up to extracting a subsequence to some $\bar\alpha_1, \bar\alpha_2$. If $\bar\alpha_1+\bar\alpha_2 >0$, since the convergence of $\frac{J_r}{|\ln(r)|}$ to $J_0$ is uniform on $[0,\pi/2]$ and $J_0$ is continuous, everything remains well defined and we can pass directly to the limit,
$$
\lim_{r\to 0 }\frac{E_r}{|\ln(r)|} = \lim_{r\to0}\frac{F_r(\alpha_{1,r},\alpha_{2,r})}{|\ln(r)|} = F_0(\bar\alpha_1, \bar\alpha_2) \le E_0
$$
and therefore,
$$ \lim_{r\to 0 }\frac{E_r}{|\ln(r)|} = F_0(\bar\alpha_1, \bar\alpha_2) = E_0.$$
Finally, from the uniqueness in Proposition~\ref{prop:maximalityF0}, we deduce $\bar\alpha_1 = \bar\alpha_2 = \frac{\pi}{6}$.

Now, if $\bar\alpha_1 = \bar\alpha_2 = 0$. By the first assertion of Proposition~\ref{prop:Convergence} we deduce, noting $\alpha_{3,r} = \pi/2-(\alpha_{1,r}+\alpha_{2,r})$, 
$$
\begin{aligned}
\lim_{r\to0} \frac1{|\ln(r)|}\frac{J_r(\alpha_{1,r})}{|\T^{\alpha_{1,r}}|} = \frac12,\\
\lim_{r\to0} \frac1{|\ln(r)|}\frac{J_r(\alpha_{2,r})}{|\T^{\alpha_{2,r}}|} = \frac12,\\
\lim_{r\to0} \frac1{|\ln(r)|}\frac{J_r(\alpha_{3,r})}{|\T^{\alpha_{3,r}}|} = 0.
\end{aligned}
$$
And as a consequence, doing similar manipulation as for proving the existence, we obtain
$$ \lim_{r\to 0 }\frac{E_r}{|\ln(r)|} =\lim_{r\to0} \frac{F_r(\alpha_{1,r},\alpha_{2,r})}{|\ln(r)|} = \frac14.$$
which is in contradiction with the bound $\liminf_{r\to0} E_r/|\ln(r)|\ge E_0 >1/4$.
\qed

\section{The Cheeger case }\label{sec:Cheeger}
In this section we prove Theorem~\ref{Th:cheeger}, partly following the same strategy as for the proof of Theorem~\ref{TH1}.\\
Let us define for $r>0$ and a triangle $\T$ of vertices $A_i$ the following Neumann Cheeger constant 
\be\label{df:cheegerTriangle}
h_{N}(\T,r) = \inf\left\{\frac{\per(E,\T)}{|E|}\middle| \ E\subset \T \setminus (V_\T\oplus B_r)\right\},
\ee
where by $\per(E,\T)$ we denote the relative perimeter of the set $E$ in the triangle $\T$.\\
As a first step, we prove the following.

In the case where $\T\sm [V_\T \oplus \ov{B_r}]$ is empty, we define $h_N(\T,r) = +\infty$.

\begin{lem}\label{lem:equtriangleCheeger}
For all $0<r<0.19$, the equilateral triangle $\Teq$ of circumradius $1$ is the unique minimizer of $h_{N}(\T,r)$ among all triangles $\T$ of circumradius $\Rcal(\T) = 1$.
\end{lem}

\subsection{Proof of Lemma~\ref{lem:equtriangleCheeger}}
First, if $\T$ has an obtuse angle and satisfies, $|\T\sm{V_\T\oplus \ov{B_r}}| > 0$, with the same symmetrization procedure as for the averaged torsion, we can find a right or acute triangle which is a better competitor. \ 

Let $0<r<0.19$ and consider $\T$ a right or acute triangle of circumradius $\Rcal(\T)=1$ such that $|\T\sm [V_\T \oplus \ov{B_r}] >0$. The Neumann Cheeger set of $\T \sm (V_\T \oplus \ov {B_r})$ exists and consists of arcs of circle inside $\T \sm (V_\T \oplus \ov {B_r})$ which touch tangentially the boundary of the balls and orthogonally the rest of the boundary. If a triangle $\T$ is such that $\T \sm (V_\T \oplus \ov {B_r})$ is a self Neumann Cheeger set then the three balls $B_r(A_i)$ do not overlap. Of course, not all triangles for which the balls do not overlap lead to a self Cheeger $\T \sm (V_\T \oplus \ov {B_r})$, however such triangles can not be optimal. Assume that an arc joins two of the balls $B_r(A_i)$ and align the triangle with these two vertices on the horizontal axis. Then
\[\begin{split}
h_{N}(\T,r)&=\inf\left\{\frac{\int _\T |\nabla u|}{\int _\T |u|} : u \in C^\infty (\ov \T), u= 0 \mbox{ on } V_\T\oplus \ov {B_r}\right\}\\
&\ge\inf\left\{\frac{\int _\T |u_y|}{\int _\T |u|} : u \in C^\infty (\ov \T), u= 0 \mbox{ on } V_\T\oplus \ov {B_r}\right\}\ge \frac 12.
\end{split}\]
The last inequality above, comes by vertical slicing of $\T$ using that if $v \in C^\infty ([0,l))$, $v(0)=0$ then $\int_0^l |v'|\ge \frac 1l \int_0^l |v|$ and any slice has length not larger than $2$. In the same time, taking as test $E= \Teq\sm V_{\Teq}\oplus \ov {B_r}$ in $\Teq \sm V_{\Teq}\oplus \ov {B_r}$, we get, since for the equilateral triangle the three balls do not overlap, 
$$h_{N}(\Teq,r)\le \frac{4\pi r}{3\sqrt{3} -2\pi r^2}<\frac 12.$$
The previous arguments imply two facts:
\begin{itemize}
\item any triangle for which the relative boundary of the Neumann Cheeger set has an arc connecting two discs can not be optimal since the value of the Neumann Cheeger constant is larger than $\frac12$;
\item $\Teq \sm (V_{\Teq}\oplus \ov B_r)$ is a self Neumann Cheeger set.
\end{itemize} 
Now, let $\T$ be a triangle of circumradius $\Rcal(\T) = 1$. If $\T \sm (V_\T \oplus \ov {B_r})$ is not a self Neumann Cheeger set, then there exists an arc connecting two discs, so its Neumann Cheeger constant is larger than $\frac 12$. If it is a self Neumann Cheeger set then its Neumann Cheeger constant equals
$$h_{N}(\T,r) = \frac{\pi r}{|\T| -\frac{\pi r^2}{2}}.$$
Since among all triangles of circumradius $1$ the one with largest area is equilateral, we conclude the proof of the Lemma.
\qed

\subsection{Proof of Theorem~\ref{Th:cheeger}}

Let $\eps >0$, and take $\O \in \A_\eps$ such that $\rho(\O) = 1$. We denote by $Z$ the set of the centers $z$ of the balls associated to $\O$ such that $ \O =\Rtwo\setminus \big(Z+\overline{B(0,\eps)}\big)$.
 
Note that we may assume, without loss of generality, that $Z$ discrete with no accumulation points. Indeed similarly to the construction used in \cite[Theorem 1.2]{BB23}, we make a decreasing outer approximation $Z_n$ of $Z$ on the square grid $\frac 1{2^n} \Z \times \frac 1{2^n} \Z$ and construct the associated sequence of sets $(\O_n)$ defined by $\O_n = \Rtwo \sm (Z_n\oplus \ov B(0,r))$. Then the sequence $(\O_n)$ converges in the hausdorff complement sense to $\O$ and therefore $h(\O_n) \to h(\O)$ and $\rho(\O_n) \to 1^-$. Let us drop the index $n$ and work with such a domain.

We define $\TT(Z)$ a Delaunay triangulation of $Z$. 

Consider now $E\subset \O$ bounded open, a competitor set for the Cheeger constant of $\O$. We denote by $\TT_E$ the subset of $\TT(Z)$ of triangles intersecting $E$, and we can make the following decomposition
$$
\frac{\per(E)}{|E|} \ge \frac{\sum_{T\in \TT_E } \per(E\cap \T,\T)}{\sum_{\T\in \TT_E } |E\cap \T|} \ge \inf_{\T\in \TT_E} \frac{\per(E\cap \T,\T)}{|E\cap \T|} \ge \inf_{\T\in \TT_E} h_N(\T,\eps),
$$
notice here that $E$ being bounded, the set $\TT_E$ is finite. We denote by $\bar \T$ the triangle acheiving the minimum. Since $\TT(Z)$ is a Delaunay triangulation, by the empty circle property, the triangle $\bar\T$ has circumradius $R(\bar\T) \le 1+\eps$ and we denote $\tilde \T = \frac1 {R(\bar\T)} \bar \T$ which has circumradius $1$. From the scaling law of the Neumann Cheeger constant and the fact that it is decreasing for the inclusion of the Dirichlet regions, we get,
$$ \frac{\per(E)}{|E|}\ge\frac1{(1+\eps)} h_N(\tilde \T, \frac{\eps}{1+\eps}) $$
And Lemma~\ref{lem:equtriangleCheeger} implies that if
\be \label{eq:defeps*}
\eps \le \eps^* = \frac{r^*}{1-r^*} \approx 0.23,
\ee
it holds
$$\frac{\per(E)}{|E|}\ge\frac1{(1+\eps)} h_N(\Teq, \frac{\eps}{1+\eps}).$$
This last inequality being true for any competitor set $E$, we can take the infimum and obtain 
$$h(\O)\ge\frac1{(1+\eps)} h_N(\Teq, \frac{\eps}{1+\eps}).$$

Let us now prove that 
$$\O^* = \Rtwo\setminus \big(\Zhex_\eps + \overline{B(0,\eps)}\big),$$
achieves the minimum.

We define the sequence $(E_n)_n$ as 
$$\forall n, \qquad E_n = B(0,n) \cap \O^*.$$
Then, we can compute 
$$
\begin{aligned}
|E_n|&=\frac{\pi n^2}{(1+\eps)^2 |\Teq|} (1+\eps)^2 \left(|\Teq|- \frac{\pi}{2}\left(\frac{\eps}{1+\eps}\right)^2\right) + O(n),\\
\per(E_n) &= \frac{\pi n^2}{(1+\eps)^2 |\Teq|} (1+\eps) \left( \pi \frac{\eps}{1+\eps}\right) + O(n),
\end{aligned}
$$
which yields 
$$h(\O^*)\le\frac{1}{1+\eps} h_N(\Teq, \frac{\eps}{1+\eps}) + O(1/n).$$

Finally, taking the limit we obtain the desired result.
\qed

\bigskip

\noindent{\bf Acknowledgments.} GB is member of the Gruppo Nazionale per l'Analisi Matematica, la Probabilit\`a e le loro Applicazioni (GNAMPA) of the Istituto Nazionale di Alta Matematica (INdAM).


\bigskip\noindent
Dorin Bucur: Laboratoire de Math\'ematiques (LAMA), Universit\'e de Savoie\\
Campus Scientifique - 73376 Le-Bourget-Du-Lac, FRANCE\\
{\tt dorin.bucur@univ-savoie.fr}\\
{\tt https://bucur.perso.math.cnrs.fr/}

\bigskip\noindent
Giuseppe Buttazzo: Dipartimento di Matematica, Universit\`a di Pisa\\
Largo B. Pontecorvo 5, 56127 Pisa - ITALY\\
{\tt giuseppe.buttazzo@unipi.it}\\
{\tt http://www.dm.unipi.it/pages/buttazzo/}

\bigskip\noindent
Alexis de Villeroch\'e: Laboratoire de Math\'ematiques (LAMA), Universit\'e de Savoie\\
Campus Scientifique - 73376 Le-Bourget-Du-Lac, FRANCE\\
{\tt alexis.devilleroche@univ-smb.fr}\\
{\tt https://alexisdevilleroche.perso.math.cnrs.fr/}

\end{document}